\documentclass[12pt]{article}
\usepackage{natbib}

\usepackage[english]{babel}

\usepackage[utf8]{inputenx}
\usepackage[T1]{fontenc}

\usepackage{rotating} 

\usepackage{amsfonts, amsmath, amssymb, amsthm} 

\usepackage[ruled,vlined,linesnumbered]{algorithm2e} 

\usepackage{graphicx}

\usepackage{color} 
\usepackage{xcolor}

\usepackage{tikz}
\usetikzlibrary{shapes,snakes}

\usepackage{multirow}

\usepackage{array}

\usepackage{calc}

\usepackage{perpage}
\usepackage{textcomp}
\usepackage{verbatim}

\usepackage{dsfont} 
\usepackage{grffile} 
\usepackage{enumitem}

\usepackage{pdfsync}

\RequirePackage{url}

\graphicspath{ {./figures/} } 

\newtheorem{lemma}{Lemma}[section]

\newtheorem{remark}{Remark}[section]

\usepackage[textwidth=4cm, textsize=footnotesize]{todonotes}

\newcommand{\todocolor}{magenta}
\tikzstyle{todoboxstyle} = [draw=\todocolor, fill=\todocolor!10, very thick,
    rectangle, rounded corners, inner sep=10pt, inner ysep=10pt]
\tikzstyle{fancytitle} =[fill=\todocolor!10, text=\todocolor]

\usepackage{color} 

\newif\ifgray
\graytrue
\grayfalse
\IfFileExists{BackgroundColor.tex}{\grayfalse
}{}
\IfFileExists{../BackgroundColor.tex}{}{}
\ifgray
\pagecolor{black!30!white}
\fi

\newcommand{\1}{{\rm 1}\mskip -4,5mu{\rm l} }

\newcommand{\argmin}{\operatornamewithlimits{argmin}}

\newcommand{\normone}[1]{\ensuremath{\!|\!| #1 | \! |_{1}}}

\newcommand{\norm}[1]{\ensuremath{\!|\!| #1 | \! |}}

\newcommand{\inprod}[2]{\ensuremath{\langle #1 , \, #2 \rangle}}

\newcommand{\suchthat}{\;\ifnum\currentgrouptype=16 \middle\fi|\;}

\newcommand{\tp}{\ensuremath{^\top}}

\def\E{\mathbb{E}}

\def\R{\mathbb{R}}

\newcommand{\Rp}{\ensuremath{\R^p}}

\newcommand{\Rpp}{\ensuremath{\R^{p\times p}}}
\newcommand{\otp}{\ensuremath{\{1,\dots,p\}}}

\newcommand{\si}{\sum_{i=1}^n} 
\def\sino{\frac{1}{n}\sum_{i=1}^n} 

\newcommand{\mcset}{\ensuremath{Y}}

\newcommand{\er}{Erd\H os-R\'enyi}

\newcommand{\blind}{0}

\usepackage{tikz-cd} 
\begin{document}
	
\def\spacingset#1{\renewcommand{\baselinestretch}%
	{#1}\small\normalsize} \spacingset{1}

\if0\blind
{
	\title{\bf Graphical Models for Discrete~and~Continuous~Data}
	\author{Rui Zhuang 
		\hspace{.2cm}\\
		Department of Biostatistics, University of Washington\\
		and \\
		Noah Simon\\
		Department of Biostatistics, University of Washington\\
		and \\
		Johannes Lederer\\
		Department of Mathematics, Ruhr-University Bochum
	}
	\maketitle
} \fi

\if1\blind
{
	\bigskip
	\bigskip
	\bigskip
	\begin{center}
		{\LARGE\bf Title}
	\end{center}
	\medskip
} \fi

\bigskip
\begin{abstract}
	We introduce a general framework for undirected graphical models. It generalizes Gaussian graphical models to a wide range of continuous, discrete, and combinations of different types of data. The models in the framework, called exponential trace models, are amenable to estimation based on maximum likelihood. We introduce a sampling-based approximation algorithm for computing the maximum likelihood estimator, and we apply this pipeline to learn simultaneous neural activities from spike data. 
\end{abstract}

\noindent%

\spacingset{1.45}


\newcommand{\originaldim}{\ensuremath{p}}
\newcommand{\newdim}{\ensuremath{q}}
\newcommand{\Rnd}{\ensuremath{\R^\newdim}}
\newcommand{\dimt}{\ensuremath{2}}
\newcommand{\Rdimt}{\ensuremath{\R^\dimt}}
\newcommand{\Rqq}{\ensuremath{\R^{\newdim\times\newdim}}}

\newcommand{\data}{\ensuremath{X}}
\newcommand{\dataa}{\ensuremath{x}}
\newcommand{\datav}{\ensuremath{{\bf x}}}
\newcommand{\datamca}{\ensuremath{{\bf z}}}
\newcommand{\datamc}{\ensuremath{{Z}}}
\newcommand{\datas}{\ensuremath{X}}
\newcommand{\datasi}{\ensuremath{{X^i}}}
\newcommand{\datasseq}{\ensuremath{{\datas^1,\dots,\datas^n}}}
\newcommand{\datasseqa}{\ensuremath{{\datav^1,\dots,\datav^n}}}

\newcommand{\otq}{\ensuremath{\{1,\dots\newdim\}}}

\newcommand{\Matrix}{\ensuremath{\mathfrak{M}}}
\newcommand{\Matrixint}{\ensuremath{\mathfrak{M^*}}}

\renewcommand{\matrix}{\ensuremath{{\mathrm M}}}
\newcommand{\matrixtrue}{\ensuremath{{\matrix}}}
\newcommand{\matrixtoy}{\ensuremath{{\widetilde\matrix}}}

\newcommand{\normtotal}{\ensuremath{\gamma}}

\newcommand{\tr}{\ensuremath{\operatorname{tr}}}
\newcommand{\fisher}{\ensuremath{\operatorname{I}_\tensortrue}}
\newcommand{\estfisher}{\ensuremath{\widehat{\operatorname{I}}_\tensortrue}}
\newcommand{\estfishermatrix}{\ensuremath{\widehat{\operatorname{I}}_\matrixtrue}}

\newcommand{\matrixest}{\ensuremath{{\widehat{\matrix}}}}

\newcommand{\normtr}[1]{\norm{#1}_{\operatorname{tr}}}
\newcommand{\inprodtr}[2]{\inprod{#1}{#2}_{\operatorname{tr}}}

\newcommand{\hterm}{\ensuremath{\xi}}
\newcommand{\htermdatav}{\ensuremath{\hterm(\datav)}}

\newcommand{\Tdatap}{\ensuremath{T}}
\newcommand{\Tdatapv}{\Tdatap(\datav)}
\newcommand{\Tdatae}{\ensuremath{\overline{T}}}

\newcommand{\Domain}{\ensuremath{\mathcal D}}

\newcommand{\measure}{\ensuremath{\nu}}

\newcommand{\nb}{\ensuremath{\mathcal N}}
\newcommand{\nbj}{\ensuremath{\mathcal \nb(j)}}
\newcommand{\bigCI}{\mathrel{\text{\scalebox{1.07}{$\perp\mkern-10mu\perp$}}}}

\newcommand{\tensor}{\matrix}

\newcommand{\tensortrue}{\matrixtrue}

\newcommand{\Tensor}{\Matrix}

 \newcommand{\tensortrans}{{\textcolor{red}{\ensuremath{\matrix}}}}

 \newcommand{\tensorest}{\matrixest}

\newcommand{\datatotala}{\underline{\datav}}
\newcommand{\datatotal}{\underline{\data}}

\newcommand{\tensorspace}{\Rqq}

\newcommand{\tensorspaceinit}{\Domain}

\newcommand{\link}{\ensuremath{e}}
\newcommand{\linkd}{\ensuremath{e}}
\newcommand{\linkinv}{\ensuremath{\log}}

\newcommand{\tensorindex}{_{ij}}
\newcommand{\tensorindexsum}{_{k,l=1}}
\newcommand{\tensorindexs}{{i,j}}
\newcommand{\tensorindexp}{_{kl}}
\newcommand{\tensorindexpsum}{_{k',l'=1}}
\newcommand{\tensorindexps}{{i'j'}}

\newcommand{\normaltensor}{\ensuremath{A}}

\newcommand{\hypn}{\ensuremath{\mathrm{H_N}}}
\newcommand{\hypa}{\ensuremath{\mathrm{H_A}}}

\newcommand{\corvec}{\ensuremath{\mathbf{m}}}

\newcommand{\corfun}{\ensuremath{\mathbf{t}}}
\newcommand{\corfunv}{\ensuremath{\corfun(\datav)}}

\newcommand{\dataona}{\ensuremath{{\bf y}}}
\newcommand{\dataonaa}{\ensuremath{{y}}}
\newcommand{\dataoffa}{\ensuremath{{z}}}
\newcommand{\datauna}{\ensuremath{{\bf w}}}
\newcommand{\dataon}{\ensuremath{{Y}}}
\newcommand{\dataoff}{\ensuremath{{Z}}}
\newcommand{\dataun}{\ensuremath{{W}}}
\newcommand{\potentialset}{\ensuremath{A}}
\newcommand{\linkfunction}{\ensuremath{\phi}}
\newcommand{\matrixo}{\ensuremath{\matrix^1}}
\newcommand{\matrixoo}{\ensuremath{\matrix^2}}
\newcommand{\matrixooo}{\ensuremath{\matrix^3}}
\newcommand{\matrixot}{\ensuremath{\matrix^{12}}}
\newcommand{\setone}{\ensuremath{\1\{\dataoffa=0\}}}
\newcommand{\setonec}{\ensuremath{\1\{\dataoffa=1\}}}

\section{Introduction}
Gaussian graphical models~\citep{Drton16,Lauritzen96,Wainwright08} describe the dependence structures in normally distributed random vectors. These models have become increasingly popular in the sciences, because their  representation of the dependencies is lucid and can be readily estimated. For a brief overview, consider a random vector~$\data\in\Rp$ that follows a centered normal distribution with density
\begin{equation}\label{gaussian}
  f_{\Sigma}(\datav)=\frac{1}{(2\pi)^{p/2}\sqrt{ |\Sigma|}}\,e^{-\datav\tp\Sigma^{-1}\datav/2}
\end{equation}
with respect to Lebesgue measure, where the population covariance $\Sigma\in\Rpp$ is a symmetric and positive definite  matrix. Gaussian graphical models associate these densities with a graph $(V,E)$ that has vertex set $V:=\{1,\dots,p\}$ and edge set $E:=\{(i,j):i,j\in\otp,i\neq j,\Sigma^{-1}_{ij}\neq 0\}.$ The graph encodes the dependence structure of~\data\ in the sense that any two entries $\data_i,\data_j,$ $i\neq j,$ are conditionally independent given all other entries if and only if $(i,j)\notin E$. A natural and straightforward estimator of $\Sigma^{-1},$ and thus of $E,$ is the inverse of the empirical covariance, which is also the maximum likelihood estimator. Inference can then be approached via the quantiles of the normal distribution.

The problem with Gaussian graphical models is that in practice, data often violate the normality assumption. Data can be discrete, heavy-tailed, restricted to positive values, or deviate from normality in other ways. Graphical models for some types of non-Gaussian data have been developed, including copula-based models~\citep{Liu15,MR3059084,Liu09,Xue12}, score matching approach~\citep{lin2016estimation,yu2019generalized}, Ising models~\citep{Brush67,Lenz20}, and multinomial extensions of the Ising models~\citep{Loh13}. However, there is no general framework that comprises different data types including finite and infinite count data, potentially heavy-tailed continuous data, and combinations of discrete and continuous data and at the same time, ensures a rigid theoretical structure.

We make three main contributions in this paper: 
\begin{itemize}
\item We formulate a general framework for undirected graphical models that both encompasses previously studied and new models for continuous, discrete, and combined data types.
\item We show that maximum likelihood is based on a convex and smooth optimization function and provides consistent estimation and inference for the model parameters in the framework.
\item We establish a sampling-based approximation algorithm for computing the maximum likelihood estimator.
\end{itemize}

Let us have a glance at the framework. For this, we start with the  Gaussian densities~\eqref{gaussian}. Our first observation is that  $-\datav\tp\Sigma^{-1}\datav/2=-\inprodtr{\Sigma^{-1}}{\datav\datav\tp/2},$ where $\inprodtr{\cdot}{\cdot}$ is the trace inner product. This formulation looks ``somewhat less revealing''~\cite[Page~125]{Eaton07} on first sight, but  it has two conceptual advantages. First, we argue that writing the parameters and the data as algebraic duals of each other makes their relationship more symmetric. Second, we argue that it is a good starting point for generalizations. For this, we take the viewpoint that the exponents in the densities are linear functions of the matrix $\datav\datav\tp/2,$ and then replace this matrix by a general matrix-valued function~\Tdatap\ of \datav. Our second observation is that the fundamental quantity in the family of  Gaussian graphical models is the inverse covariance matrix~$\Sigma^{-1}$ rather than the covariance matrix~$\Sigma$ itself. This suggests a reparametrization of the model using the matrix~$\Sigma^{-1},$ which is then replaced by a general matrix~\matrix. This subtlety is important: as we will see later, the matrix~\matrix\ contains all information about the dependence structure of~\data, while the equality of $\matrix^{-1}$ and the covariance matrix is a mere coincidence in the Gaussian case. With these two observations in mind, and denoting  the log-normalization by $\normtotal(\matrix)$,  with $\normtotal(\matrix)=\log((2\pi)^{p}|\matrix^{-1}|)/2$ in the Gaussian case, we can then generalize the densities~\eqref{gaussian} to
\begin{equation*}
  f_{\matrix}(\datav)=e^{-\inprodtr{\matrix}{\Tdatap(\datav)}-\normtotal{(\matrix)}}
\end{equation*}
with respect to an arbitrary $\sigma$-finite measure $\measure$, and with $\matrix,\Tdatap\in\Rqq$ a matrix-valued parameter and  data function, respectively. While additional data terms can be absorbed in the measure~\measure, it is sometimes illustrative to write them  explicitly. We thus consider distributions with densities of the form
\begin{equation*}
  f_\tensor(\datav)=e^{-\inprodtr{\tensor}{\Tdatap(\datav)}+\hterm\left(\datav\right)-\normtotal(\matrix)}\,,
\end{equation*}
where $\hterm(\datav)$ depends on~$\datav$ only. These densities form an exponential family indexed by~\matrix\ and are called exponential trace models in the following for convenience.

We recall that the well-known pairwise interaction models can also be written in exponential form, but there are important differences to the above formulation. First,  pairwise interaction models and exponential trace models are not sub-classes of one another: exponential trace models are not limited to pairwise interactions, while pairwise interaction models are not limited to canonical parameterizations. Yet, importantly, exponential trace models generalize pairwise interaction models in the sense that all generic examples of pairwise interaction models are encompassed. We also highlight that in contrast to pairwise interaction models, we allow for $\newdim\neq \originaldim,$ which  helps for concise formulations of mixed graphical models, for example. In general, we argue that the exponential trace framework is a practical starting point for general studies of graphical models, because it comprises a very large variety of examples and still ensures a firm theoretical structure.

We carefully specify and study the  described distributions in the following sections. 
Section~\ref{sec:framework} contains the proposed framework: we define the densities in Section~\ref{sec:model}, and we discuss a variety of examples in Sections~\ref{sec:examplesstandard} and~\ref{sec:examplesnew}. 
Section~\ref{sec:estimation} is focused on estimation: we discuss the maximum likelihood estimator for the model parameters~in Section~\ref{sec:estimator}, and we introduce a numerical algorithm in Sections~\ref{sec:algorithm}. 
Section~\ref{sec:simulation} shows simulation results for different types of data. Section~\ref{sec:application} applies the method to neural spike data. 
We conclude the paper with a discussion in Section~\ref{sec:discussion}. 
The proofs are deferred to the Appendix.


\subsection*{Notation}  For matrices $A,B\in\R^{s\times t},$ $s,t\in\{1,2,\dots\}$, we denote the trace inner product (or Frobenius inner product) by
\begin{align*}
  \inprodtr{ A}{ B}:=\tr\big(A\tp B\big)=\sum_{i=1}^s\sum_{j=1}^t A_{ij} B_{ij}
\end{align*}
and the corresponding norm by
\begin{align*}
  \normtr{A}:=\sqrt{\inprodtr{A}{ A}}=\sqrt{\sum_{i=1}^s\sum_{j=1}^tA_{ij}^2}\ .
\end{align*}

We consider random vectors $\data=(\data_1,\dots,\data_p)\tp\in\R^p.$ We denote random vectors and their realizations by upper case letters such as~$\data$ and arguments of functions by lower case, boldface letters such as~\datav.   Given a set $S\subset\otp,$ we denote by $\data_{S}\in\R^{|S|}$ the vector that consists of the coordinates of~$\data$ with indices in~$S,$ and we set $\data_{-S}:=\data_{S^c}\in\R^{p-|S|}$. Independence of two elements $\data_i$ and $\data_j$, $i\neq j,$ is denoted by $\data_i\perp \data_j$; conditional independence of $\data_i$ and $\data_j$ given all other elements is denoted by $\data_i\bigCI \data_j|\data_{-\{i,j\}}.$


\section{Framework}\label{sec:framework}
We first discuss our framework. In Section~\ref{sec:model}, we formulate the densities. In Section~\ref{sec:examplesstandard}, we show that these densities apply to standard examples of graphical models. In Section~\ref{sec:examplesnew}, we  study additional examples.

\subsection{Exponential Trace Models}\label{sec:model} In this section, we formulate  probabilistic models for vector-valued observations that have dependent coordinates. Specifically, we consider   arbitrary (non-empty) finite or continuous domains $\Domain\subset\R^p$ and random vectors~$\data\in\Domain$ that have densities of the form
\begin{equation}\label{model}
f_\tensor(\datav):=e^{-\inprodtr{\tensor}{\Tdatap(\datav)}+\hterm\left(\datav\right)-\normtotal(\matrix)}
\end{equation}
with respect to some $\sigma$-finite measure \measure\ on \Domain. For reference, we call these models exponential trace models.

We begin by  specifying the different components of our model. The densities are indexed by $\tensor\in\Tensor$, where $\Tensor$ is a subset of \begin{equation*}
  \Matrixint:=\operatorname{interior}\{\matrix\in\Rqq:\normtotal(\matrix)<\infty\}\,.
\end{equation*}
Unlike conventional frameworks, we do not require $q=p,$ with the advantage of concise formulations of mixture models, for example.  In generic applications, \matrix\   comprises the dependence structure of~\data\ and determines if two coordinates $\data_i$ and $\data_j$ are positively or negatively correlated. We will discuss these aspects in the next sections. The arguably most important note here is that the integrability condition $\normtotal(\matrix)<\infty$  is feasible. In particular, our framework provides natural formulations of models that avoid unreasonable restrictions on the parameter space. It is best to see this in specific examples, so that we defer to later.

Next, the data enters the model via a matrix-valued function
\begin{align*}
  \Tdatap:\Domain\to&\ \tensorspace\\
  \datav\mapsto&\ \Tdatap(\datav)
\end{align*}
and a  real-valued function
\begin{align*}
  \hterm:\Domain\to&\ \R\\
  \datav\mapsto&\ \hterm(\mathbf{\datav})\,, 
\end{align*}
and $\normtotal(\tensor)$ is the normalization defined as
\begin{equation*}
\normtotal(\tensor):=\log\int_\Domain e^{-\inprodtr{\tensor}{\Tdatap(\datav)}+\hterm(\datav)}\,d\measure\,.
\end{equation*}

We finally have to impose two technical assumptions on the parameter space. Our first assumption is that the function $\matrix\mapsto f_\matrix$ of \Matrix\ to the densities with respect to the measure~\measure\ is bijective. Sufficient conditions for this are provided in \citep[Page~199]{Berk72} and \citep[Definition~1.3]{MR558392}; we stress, however, that the bijection is required here only on~\Matrix\ rather than on the full set~\Matrixint.   Our second assumption is that  $\Matrix$ is convex and that $\Matrix$ is open with respect to an affine subspace of~$\Rqq.$ The two assumptions ensure, in particular, that the parameter~$\matrix$ is identifiable and  has a compact and ``full-dimensional'' neighborhood in an affine subspace of \Rqq. Importantly, however, the assumptions are mild enough to allow for overparametrizations in the sense that $\Matrix$ does not have to be open in $\Rqq;$ a typical example is $\Matrix$ equal to the set of symmetric matrices, seen in Sections~\ref{sec:examplesstandard}~and~\ref{sec:examplesnew}.

The exponential family formulation equips the framework with desirable structure. In particular, we can derive the following.
\begin{lemma}\label{propertiesfamily}
  The following two properties are satisfied. 
  \begin{enumerate}
  \item The set \Matrixint\ is convex;
  \item for any $\matrix\in\Matrixint$, the coordinates of~$\Tdatap(\data)$ have moments of all orders with respect to $f_\matrix$.
  \end{enumerate}
\end{lemma}
\noindent Property 1 ensures that $\Matrix=\Matrixint$ satisfies the above assumption about $\Matrix$ and Property~2 ensures concentration of the maximum likelihood estimator discussed later.

\subsection{Standard Graphical Models}\label{sec:examplesstandard} The goal of this section is to demonstrate that standard graphical models, such as  Ising models with binary and $m$-ary responses and Gaussian and non-paranormal graphical models, fit the exponential trace framework. In combination with the results in Section~\ref{sec:estimation}, this shows in particular that standard graphical models are automatically equipped with more structure than suggested by common pairwise interaction formulations.

For the standard models, it is sufficient to consider $\newdim=\originaldim,$  $\Tdatap\tensorindex(\datav)  \equiv \Tdatap\tensorindex(\dataa_{i},\dataa_{j}),$ and $ \hterm(\datav) = \sum_{j=1}^p\hterm_{j}(\dataa_j).$  Given~\matrix, we then define a  graph  $G:=(V,E),$ where  $V:=\otp$ is the vertex (or node) set and $E:=\{(i,j):i,j\in V,i\neq j,\tensor_{ij}\neq 0\}$  the edge set. Two matrices $\tensor,\tensor'\in\Tensor\subset\Rpp$ correspond to the same graph if and only if their non-zero patterns are the same. In view of the examples below, we are particularly interested in symmetric dependence structures, that is, we consider symmetric matrices~$\matrix$ in what follows. Then, also the edge set $E$ is symmetric, that is, $(i,j)\in E$ if and only if $(j,i)\in E,$ and the graph~$G$ is called~{undirected.} 

The corresponding models~$f_\matrix$ are a special case of pairwise interaction models  (pairwise Markov networks). Assuming that $\measure$ is a product measure, a density $h$ with respect to \measure\ is a pairwise interaction model if it can be written in the form
\begin{equation*}
  h(\datav)=\prod_{i,j=1}^\originaldim h_{ij}(\dataa_i,\dataa_j)
\end{equation*}
with positive functions $h_{ij}.$ This means that the densities of pairwise interaction models can be written as products of terms that depend on at most two coordinates.

The graph~$G$ now encodes  the conditional dependence structure, as one can show by applying the Hammersley-Clifford theorem~\citep{Besag74,Grimmett73}. The theorem implies that for a strictly positive density~$h(\datav)$ with respect to a product measure, two elements $\data_i,\data_j$ are conditionally independent given all other coordinates if and only if we can write~$h(\datav)=h^1(\datav_{-i})h^2(\datav_{-j}),$ where $h^1,h^2$ are positive functions.  For the described densities in our framework, we can write $f_\tensor(\datav)=f_\tensor^1(\datav_{-i})f_\tensor^2(\datav_{-j})$ with positive functions~$f_\tensor^1,f_\tensor^2$ if and only if $\tensor_{ij}=0.$ By the above definition of the graph associated with~\tensor, the latter is equivalent to $(i,j)\notin E.$ We thus find
\begin{equation*}
  \data_i\bigCI \data_j|\data_{-\{i,j\}}\text{~~~~~if and only if~~~~~}(i,j)\notin E\,,
\end{equation*}
meaning that the conditional dependence structure of \data\ is represented by  the edge set~$E.$

The graph~$G$ also determines the unconditional dependence structure. To illustrate this, we define~$\overline E$ as the set of all connected components in $\tensor,$ that is, $(i,j)\in \overline E$ if and only if there is a path $(i,i_1),(i_1,i_2)\dots,(i_q,j)\in E$ that connects $i$ and $j$ (in particular, $E\subset \overline E$). For pairwise densities in our framework, it  is easy to check  that
\begin{equation*}
  \data_i\perp \data_j\text{~~~~~if and only if~~~~~}(i,j)\notin \overline E\,.
\end{equation*}
Hence, the dependence structure of \data\ is captured by $\overline E.$ We give an illustration of these properties in Figure~\ref{fig:graphpairwise}.


\begin{figure}[t]
  \centering
  \begin{minipage}[h]{0.48\linewidth}
  \begin{align*}
M=    \begin{bmatrix}
    1.2 & - 0.2 & 0 & 0  \\
    - 0.2 & 1.5 & 0 &  0.1  \\
     0& 0 & 1 &  0 \\
    0 & 0.1 & 0 & 0.5
\end{bmatrix}
  \end{align*}   
  \end{minipage}
  \begin{minipage}[t]{0.48\linewidth}
\begin{tikzcd}[%
    ,cells={nodes={circle,draw,font=\sffamily\footnotesize\bfseries}}
    ,every arrow/.append style={-,thick}
    ]
& 
1\arrow{dl}{}
& \\
2
& & 
4\arrow{ll}{} \\
& 
3
& 
\end{tikzcd}
  \end{minipage}
  \caption{Example of a matrix $\matrix$ (left) and the corresponding graph $G$ (right). The vertex set of the graph is $V=\{1,2,3,4\},$ the edge set of the graph is $E=\{(1,2),(2,1),(2,4),(4,2)\}$. The enlarged edge set is $\overline E=E \cup\{(1,4),(4,1)\}.$ For example, $\data_1$ and $\data_4$ are dependent (indicated by $(1,4)\in\overline E$), but they are conditionally independent given the other elements of~\data\ (indicated by $(1,4)\notin E$). }
  \label{fig:graphpairwise}
\end{figure}

We can now turn to the examples.

\subsubsection{Counting Measure}
We first describe two cases where the base measure~\measure\ is the counting measure on~$\{0,1,\dots\}^{p}.$ Specifically, we show that the well-known Ising and multinomial Ising models are encompassed by our framework.

Beyond the measure, a unifying property of the two examples is that  the integrability condition is satisfied for all matrices. In a formula, this reads
\begin{equation}\label{matricescount}
\normtotal(\matrix)<\infty\text{~~for all }\matrix\in\R^{p\times p}\,.
\end{equation}
Since the set of all matrices in $\Rpp$ is open, the integrability property~\eqref{matricescount} means that
\begin{equation*}
\Matrixint= \{\matrix\in\R^{p\times p}\}\,.
\end{equation*}
This ensures that any convex and (relatively) open set~$\Matrix\subset\Rpp$ meets our technical assumptions. Importantly, we show later that the same properties are shared by exponential trace models for Poisson data. 

\paragraph{Ising}  The Ising model has a variety of applications, for example, in Statistical Mechanics and Quantum Field Theory~\citep{Gallavotti13,Zuber77}. Its densities are proportional to
\begin{equation*}
   e^{\sum_{j=1}^p a_{jj}\dataa_{j}+\sum_{j,k}a_{jk}\dataa_{j}\dataa_{k} }
\end{equation*}
with $a_{jk}=a_{kj},$ and the domain is~$\dataa_{1},\dots,\dataa_{p}\in\{0,1\}.$ As an illustration, consider a material that consists of~$p$ molecules with one ``magnetic'' electron each. The binary variable~$\dataa_j$ then corresponds to the electron's spin (up or down in a given direction) in the $j$th molecule, and the factor $a_{jk}$ determines whether the spins of the electrons in the $j$th and  $k$th molecule  tend to align ($a_{jk}>0$ as in  ferromagnetic materials) or tend to take opposite directions ($a_{jk}<0$ as in anti-ferromagnetic materials). The Ising model is a special case of our framework~\eqref{model}. Indeed, since $1^2=1$ and $0^2=0,$ we can obtain the desired densities by setting $\Domain=\{0,1\}^p$, $\newdim=\originaldim,$ $\matrix_{ij}=-a_{ij},$ $\Tdatap_{ij}(\datav)=\dataa_i\dataa_j,$ and $\hterm\equiv0.$ Since the domain is finite, the integrability condition~$\normtotal(\matrix)<\infty$ is naturally satisfied for all matrices~\matrix.

\paragraph{Multinomial Ising} The spin of spin~$1/2$ particles (such as the electron) can take two values. Thus, the corresponding measurements can be represented by the domain~$\{0,1\}$ as in the Ising model discussed above. In contrast, the spin of spin~$s$ particles with $s\in\{1,3/2,\dots\}$ (such as $W$ and $Z$ vector bosons and composite particles) can take $2s+1$ values, which can not be represented by binaries directly. Therefore, the multinomial Ising model, which extends the Ising model to multinomial domains, is of considerable interest in quantum physics. For similar reasons, multinomial Ising models are of interest in other fields, see~\citep{Diesner05} for an example in sociology.  

We now show that also the multinomial Ising model is a special case of our framework. For this, we encode the data in an enlarged binary vector. We first denote the original multinomial data by~$Y\in\{0,\dots,m-1\}^l.$ This could correspond to $l$ spin~$(m-1)/2$ particles. We then represent the data with an enlarged binary vector $\data\in\{0,1\}^p$, $p=l\cdot (m-1),$ by setting 
\begin{equation*}
\data_i=\1\Bigl\{Y_j=i-(m-1)\lfloor\frac{i-1}{m-1}\rfloor,j=\lceil\frac{i}{m-1}\rceil\Bigr\}\in\{0,1\}\,.  
\end{equation*}
Each coordinate of the original data (for example, each spin) is now represented by $m-1$ binary variables. We can use the same model as above, except for imposing the additional requirement $\matrix_{ij}=0\text{~if~}\lceil\frac{i}{m-1}\rceil=\lceil\frac{j}{m-1}\rceil$ to avoid self-interactions. These settings yield the standard extension of the Ising model to multinomial data, cf.~\citep{Loh13}. In particular, for $m=2,$  we recover the Ising model above. Again, since the domain is finite, the integrability condition~$\normtotal(\matrix)<\infty$ is naturally satisfied for all matrices~\matrix.

\subsubsection{Continuous Measure}
We now consider two examples in which the base measure~\measure\ is the Lebesgue measure on~$\R^p.$ Specifically, we show that  Gaussian and non-paranormal  graphical models are encompassed by our framework.

We show that the integrability condition is satisfied for all matrices that are positive definite.  In a formula, this reads
\begin{equation}\label{matricescont}
\normtotal(\matrix)<\infty\text{~~for all positive definite }\matrix\in\R^{p\times p}\,.
\end{equation}
Since the set of all positive definite matrices in $\Rpp$ is open, this means that
\begin{equation*}
\Matrixint\supset \{\matrix\in\R^{p\times p}:\matrix\text{ positive definite}\}\,.
\end{equation*}
One can thus take any convex and (relatively) open set of positive definite  matrices as parameter space~\Matrix. Later, we will show that very similar models can also be formulated for exponential data, for example.

\paragraph{Gaussian} The most popular examples are Gaussian graphical models~\citep{Lauritzen96}. For centered data (which can be assumed without loss of generality), these models correspond to random vectors $\data\sim\mathcal N_p(0,\Sigma),$ where $\Sigma$ (and thus also $\Sigma^{-1}$) is a symmetric, positive definite matrix. We can generate these models in our framework~\eqref{model} by setting $\Domain=\Rp$, $\newdim=\originaldim,$ $\matrix=\Sigma^{-1},$ $\Tdatap_{ij}(\datav)=\dataa_i\dataa_j/{2}$, and $\hterm\equiv0.$ 

Gaussian graphical models are  well-studied and are also the starting point for our contribution. It is thus especially important to disentangle general properties of graphical models in our framework from  peculiarities  of the Gaussian case. First, the correspondence of the conditional independence structure and the non-zero pattern of the inverse covariance matrix is specific to the Gaussian case. This has already been pointed out in~\citep{Loh13}, where relationships between the dependence structure and the inverse covariance matrix in Ising models and other exponential families with additional interaction terms are studied. However, instead of concentrating on possible connections, we argue that it is important to distinguish clearly between the two concepts. In our framework, the (conditional) dependencies are completely captured by the matrix~\matrix. It is thus reasonable to consider~\matrix\ the fundamental quantity. Instead, the equality of $\matrix^{-1}$ and $\E_\tensor[\data\data\tp]$ is specific to the Gaussian case, and the (generalized) covariances $\E_\tensor[\data\data\tp]$ ($2\E_\tensor \Tdatap(\data)$) and their inverse should be viewed only as  a characteristic of the model. Second, the type of the node conditionals can change once dependencies are introduced. In the Gaussian case, the node conditionals are normal irrespective of~\matrix. In most other cases, however, the types of node conditionals {cannot} be the same in the dependent and independent case - unless additional assumptions are introduced. We will discuss this in the later examples below.

\paragraph{Non-paranormal} Well-known generalizations of Gaussian graphical models are non-paranormal graphical models~\citep{Liu09}. These models correspond to vectors~\data\ such that \\
$(g_1(\data_1),\dots,g_p(\data_p))\tp\sim\mathcal N_p(0,\Sigma)$ for real-valued, monotone, and differentiable functions $g_1,\dots,g_p$ and symmetric, positive definite matrix~$\Sigma$. These models can be generated in our framework by setting $\Domain=\Rp$, $\newdim=\originaldim,$ $\matrix=\Sigma^{-1}$, $\Tdatap_{ij}(\datav)=g_i(\dataa_i)g_j(\dataa_j)/{2}$, and $\htermdatav=\sum_{j=1}^p\log|g'_j(\dataa_j)|$, cf.~\citep[Equation~(2)]{Liu09}. The normalization constant is (irrespective of symmetry) $\normtotal(\matrix)=\log((2\pi)^{p}|\matrix^{-1}|)/2<\infty$ both in the Gaussian and the non-paranormal case, so that in both cases, property~\eqref{matricescont} is satisfied.

\subsection{Non-standard Examples}\label{sec:examplesnew} The idea that Ising models as well as other standard graphical models can be written as exponential families is not new~\citep[Chapter~3.3]{Wainwright08}. However, we argue that the details of the  notions and formulations are essential, especially when it comes to establishing models for data that are not covered by standard graphical models.  We will outline this in the following. We first discuss Poisson and exponential distributions. In particular, we establish thorough proofs for the integrability of the square-root models in~\citep{Inouye16} and introduce extensions that show that the square-root is just one out of many possible operations for the interaction terms and that a variety of  distributions besides Poisson can be handled. 

\paragraph{Poisson} A main objective in systems biology is the inference of microbial interactions. The corresponding data is multivariate count data with infinite range~\citep{Faust12}. Other fields where such data is prevalent include particle physics (radioactive decay of  particles) and criminalistics (number of crimes and arrests). However, copula-based approaches to infinite count data inflict severe identifiability issues~\citep{Genest07}, while standard extensions of the independent case lead to integrability issues, see below.  Multivariate Poisson data has thus obtained considerable attention in the recent Machine Learning literature~\citep{Inouye16,Inouye15,Yang13,Yang15}, but  much less in statistics.

We show in the following that within framework~\eqref{model}, one can solve the problems associated with the standard approaches  while preserving the Poisson flavor of the individual coordinates, especially  in the limit of small interactions. For this, we use our framework with the specifications $\Domain=\{0,1,\dots\}^p,$ $\newdim=\originaldim,$  $\htermdatav=-\sum_{j=1}^p\log(\dataa_j!),$  and functions~\Tdatap\ that satisfy $\Tdatap_{ii}(\datav)=\Tdatap_{ii}(\dataa_i)=\dataa_i$ and $\Tdatap_{ij}(\datav)=\Tdatap_{ij}(\dataa_i,\dataa_j)\leq c(\dataa_i+\dataa_j)$ for some $c\in(0,\infty) $.  We note that the case $\Tdatap_{ij}(\datav)=\sqrt{\dataa_i\dataa_j}$ has been introduced in the Machine Learning literature~\citep{Inouye16}, but the technical aspects of this case have not been studied, and the general setting has not been formulated altogether.

Most importantly, we need to show property~\eqref{matricescount}. To this end, we have to verify that for all matrices~$\matrix\in\Rpp$
\begin{equation*}
  \sum_{\dataa_1,\dots,\dataa_p=0}^\infty \frac{e^{-\sum_{j}\matrix_{jj}\dataa_j-\sum_{i,j:i\neq j}\matrix_{ij}\Tdatap_{ij}(\dataa_i,\dataa_j)}}{\prod_{j}\dataa_j!}<\infty\,.
\end{equation*}
Since $-\matrix_{ij}\Tdatap_{ij}(\dataa_i,\dataa_j)\leq \tilde c\dataa_i+\tilde c\dataa_j$, where $\tilde c:=c\max_{i,j}|\matrix_{ij}|,$ a sufficient condition is
\begin{equation*}
  \sum_{\dataa_1,\dots,\dataa_p=0}^\infty \frac{e^{-\sum_{j}\left(\matrix_{jj}-2p\tilde c\right)\dataa_j}}{\prod_{j}{\dataa_j}!}<\infty\,.
\end{equation*}
Hence, defining $C(\matrix,j):=e^{-\left(\matrix_{jj}-2p\tilde c\right)}\in(0,\infty),$ $j\in\{1,\dots,p\},$ the integrability condition is implied for any~\matrix\ by the fact
\begin{equation*}
  \sum_{\dataa_1,\dots,\dataa_p=0}^\infty\ \prod_{j=1}^p\frac{C(\matrix,j)^{\dataa_j}}{\dataa_j!}=\prod_{j=1}^pe^{C(\matrix,j)}<\infty\,.
\end{equation*}
This proves property~\eqref{matricescount}.

In contrast, the corresponding integrability conditions in standard approaches to this data type inflict severe restrictions on the parameter space. To see this, recall  that the joint density of $p$ independent Poisson random variables with parameters $a_1,\dots,a_p>0$  is proportional to
\begin{align*}
  \exp\Big(\sum_{j=1}^p\log(a_j)\dataa_j-\sum_{j=1}^p\log(\dataa_j!)\Big)\,.
\end{align*}
The standard approach to include interactions is to add terms of the form $a_{ij}\dataa_i\dataa_j.$ This yields densities proportional to
\begin{align}\label{naive}
  \exp\left(\sum_{j=1}^p\log(a_j)\dataa_j+\sum_{i\neq j}a_{ij}\dataa_i\dataa_j-\sum_{j=1}^p\log(\dataa_j!)\right).
\end{align}
 The dominating terms in this expression are the interaction terms $a_{ij}\dataa_i\dataa_j$. Using Stirling's approximation, we find that $x^2/\log(x!)\to \infty$ for $x\to\infty,$ showing that the density cannot be normalized unless $a_{ij}\leq 0$ for all $i,j.$ This means that the standard approach excludes positive interactions between the nodes.

Let us finally look at the node conditionals for a specific~\Tdatap. We choose $\Tdatap_{ij}(\datav)=\sqrt{\dataa_i\dataa_j}$ for simplicity. The node conditionals become
\begin{equation*}
  f_\matrix(\dataa_j|\datav_{-j})\sim e^{ -\matrix_{jj}\dataa_j-\log(\dataa_j!)}L_{\text{Int}},
\end{equation*}
where
\begin{align*}
  L_{\text{Int}}&=e^{-\sqrt{\dataa_j}\sum_{\substack{k\in\nbj}}(\matrix_{jk}+\matrix_{kj})\sqrt{\dataa_{k}}}\,.
\end{align*}
The off-diagonal terms in~\matrix\ model the interactions of $j$ with the other nodes. If the factors~$\matrix_{jk}$ are small, $L_{\text{Int}}\approx 1,$ and thus, the node~$j$ approximately follows a Poisson distribution with parameter~$e^{-\matrix_{jj}}.$ In particular, if $\matrix$ is diagonal, the nodes are independent Poisson distributed random variables. 

In comparison, the standard approach represented by Display~\eqref{naive} results in exact Poisson node conditionals for any non-positive correlations. Conversely, it has been shown in~\citep[Proposition~1 and Lemma~1]{Chen14} that one can find a distribution with  Poisson (or exponential) node conditionals {\it only} if all interactions are non-positive. Thus, an unavoidable price for ``pure'' node conditionals is a strong, in practice typically unrealistic assumption on the parameter space. Our framework avoids this assumption and is still close to the exact Poisson (exponential) distributions if the interactions are small.

\paragraph{Exponential} The exponential case is  the counterpart of the Poisson case discussed above. In particular, the standard approach to correlated exponential data is confronted with the same integrability issues as above, while approaches via framework~\eqref{model} easily satisfy the integrability conditions. 

To model exponential data, we consider $\Domain=[0,\infty)^p,$ $\newdim=\originaldim,$  and $\hterm\equiv 0.$ Again a number of transformations~\Tdatap\ would have the desired properties; however, to avoid digression, we only consider the square-root transformations   $\Tdatap_{ij}(\datav)=\sqrt{\dataa_i\dataa_j}$ that correspond to~\citep{Inouye16}; generalization are possible along the same lines as in the Poisson case. We can now check the integrability condition~\eqref{matricescont}. Denoting the smallest $\ell_2$-eigenvalue of~\matrix\ by $\kappa(\matrix)>0,$ we find
\begin{align*}
  e^{\normtotal(\matrix)}=&\int_0^\infty\cdots\int_0^\infty e^{-\sum_{i,j=1}^p\matrix_{ij}\sqrt{\dataa_i\dataa_j}}\ d\dataa_1\dots d\dataa_p\\
\leq &\int_0^\infty\cdots\int_0^\infty e^{-\kappa(\matrix)\,\normone{\datav}}\ d\dataa_1\dots d\dataa_p\\
=&\Big(\int_0^\infty e^{-\kappa(\matrix)\dataa}d\dataa\Big)^p\\
=&\,\kappa(\matrix)^{-p}\ <\infty\,.
\end{align*}
Hence, $\normtotal(\matrix)<\infty$ for all positive definite matrices~\matrix.

In contrast, adding linear interaction terms to the independent joint density forbids positive correlations.  One can check this similarly as in the Poisson case above.

The node conditionals finally become
\begin{equation*}
  f_\matrix(\dataa_j|\datav_{-j})\sim e^{ -\matrix_{jj}\dataa_j}L_{\text{Int}}\,,
\end{equation*}
where
\begin{align*}
  L_{\text{Int}}&=e^{-\sqrt{\dataa_j}\sum_{\substack{k\in\nbj}}(\matrix_{jk}+\matrix_{kj})\sqrt{\dataa_{k}}}\,.
\end{align*}
The off-diagonal terms in~\matrix\ model the interactions of $j$ with the other nodes. If the factors~$\matrix_{jk}$ are small, $L_{\text{Int}}\approx 1,$ and thus, node~$j$ approximately follows an exponential distribution with parameter~$\matrix_{jj}.$ In particular, if $\matrix$ is diagonal, all node conditionals follow independent exponential distributions.

\paragraph{Composite Models} As an example for composite models, let us consider data with Poisson and exponential elements. Note first that the conditions on the set of matrices~\Matrix\ are different in the discrete examples and the continuous examples: In the discrete examples, we have shown $\gamma(\matrix)<\infty$ for any matrix~\matrix. In the continuous examples, we have shown $\gamma(\matrix)<\infty$ under the additional assumption that~\matrix\ is positive definite. In the case of composite models, one can interpolate the conditions. However, for the sake of simplicity, we instead assume that the matrices~\matrix\ are positive definite. We then consider  $p_1,p_2\in\{1,2,\dots\}$, $p_1+p_2=p,$ $\Domain=\{0,1,\dots\}^{p_1}\times [0,\infty)^{p_2}$, $\Matrix\subset\{\matrix\in\R^{p\times p}:\matrix\text{~symmetric, positive definite}\}$, $\Matrix$ open and convex, and $\htermdatav=-\sum_{j=1}^{p_1}\log(\dataa_j!)$. Hence, the first~$p_1$ elements of the random vector $\data$ are discrete, while the other $p_2$ elements are continuous. Using again the square-root transformation, the Poisson-type node conditionals for $j\in\{1,\dots,p_1\}$ are
\begin{equation*}
  f_\matrix(\dataa_j|\datav_{-j})\sim e^{ -\matrix_{jj}\dataa_j-\log(\dataa_j!)}L_{\text{Int}}\,,
\end{equation*}
where
\begin{align*}
  L_{\text{Int}}&=e^{-\sqrt{\dataa_j}\sum_{\substack{k\in\nbj}}(\matrix_{jk}+\matrix_{kj})\sqrt{\dataa_{k}}}\,.
\end{align*}
The exponential-type node conditionals for $j\in\{p_1+1,\dots,p_1+p_2\}$ have the corresponding form. The expressions highlight that the densities can include interactions between the discrete and continuous elements of~\data, while the Poisson/exponential-flavors of the nodes are still preserved.

To show that $\gamma(\matrix)<\infty,$ we proceed similarly as in the examples above. More precisely, denoting  the smallest $\ell_2$-eigenvalue of $\matrix$ by $\kappa(\matrix)>0$, we find
\begin{align*}
  e^{\normtotal(\matrix)}=&  \sum_{\dataa_1,\dots,\dataa_{p_1}=0}^\infty\ \int_0^\infty\cdots\int_0^\infty \frac{e^{-\sum_{i,j=1}^p\matrix_{ij}\sqrt{\dataa_i\dataa_j}}}{\prod_{j=1}^{p_1}\dataa_j!}\ d\dataa_{p_1+1}\dots d\dataa_{p}\\
\leq &\sum_{\dataa_1,\dots,\dataa_{p_1}=0}^\infty\ \int_0^\infty\cdots\int_0^\infty \frac{e^{-\kappa(\matrix)\,\normone{\datav}}}{\prod_{j=1}^{p_1}\dataa_j!}\ d\dataa_{p_1+1}\dots d\dataa_p\\
= &\sum_{\dataa_1,\dots,\dataa_{p_1}=0}^\infty\ \frac{e^{-\kappa(\matrix)\,(\dataa_1+\dots+\dataa_{p_1})}}{\prod_{j=1}^{p_1}\dataa_j!}\\
&~~~~~\times \int_0^\infty\cdots\int_0^\infty e^{-\kappa(\matrix)\,(\dataa_{p_1+1}+\dots+\dataa_{p_2})}\ d\dataa_{p_1+1}\dots d\dataa_p\\
=&\Big(e^{e^{-\kappa(\matrix)}}\Big)^{p_1}\Big(\int_0^\infty e^{-\kappa(\matrix)\dataa}d\dataa\Big)^{p_2}
<\infty\,.
\end{align*}

\paragraph{Laplace and Beyond} There is much room for our creativity in constructing models. For example, we can readily establish models for Laplace (double-exponential) data by inserting absolute values throughout, for example, $\Tdatap_{ij}(\datav)=\sqrt{|\dataa_i\dataa_j|}.$ Indeed,  again denoting the smallest $\ell_2$-eigenvalue of~\matrix\ by $\kappa(\matrix)>0,$ we find
\begin{align*}
  e^{\normtotal(\matrix)}=&\int_{-\infty}^\infty\cdots\int_{-\infty}^\infty e^{-\sum_{i,j=1}^p\matrix_{ij}\sqrt{|\dataa_i\dataa_j|}}\ d\dataa_1\dots d\dataa_p\\
\leq &\int_{-\infty}^\infty\cdots\int_{-\infty}^\infty e^{-\kappa(\matrix)\,\normone{\datav}}\ d\dataa_1\dots d\dataa_p\\
=&\Big(2\int_{0}^\infty e^{-\kappa(\matrix)\dataa}d\dataa\Big)^p\\
=&\,(2/\kappa(\matrix))^{p}\ <\infty\,.
\end{align*}
Hence, $\normtotal(\matrix)<\infty$ for all positive definite matrices~\matrix.  In general, the essentially only limit is that to ensure integrability, the interaction terms have to be of ``smaller order'' than the terms that correspond to the independent case. 

\section{Estimation}\label{sec:estimation}We now turn to estimation based on maximum likelihood. 
In Section~\ref{sec:estimator}, we show that maximum likelihood estimation has desirable properties in our framework. 
In Section~\ref{sec:algorithm}, we propose a sampling-based approximation algorithm for computing the maximum likelihood estimator.

\subsection{Maximum Likelihood Estimation}\label{sec:estimator} We study maximum likelihood estimation in our framework. For this, we assume we are given $n$ i.i.d.\@ data samples $\datasseq$ from a distribution of the form~\eqref{model}. Also, we assume we are given the model specifications~$\Domain,\measure,\Tdatap,\hterm$ and the parameter space $\Matrix$; that is, we assume that the model class is known. In contrast, the correct model parameters specified in the matrix~\matrixtrue\ are unknown. Our goal is to estimate~\matrixtrue\ from the data.

As a toy example, consider data on $p$ different populations of freshwater fish in $n$ similar lakes. More specifically, consider vector-valued observations $\datasseq\in\{0,1,\dots\}^p$, where $(\data^i)_j$ is the number of fish of type $j$ in lake $i$. We want to use these data to uncover the relationships among the different populations. A model suited for this task is the Poisson model discussed earlier. For example, we might set $\Domain=\{0,1,\dots\}^p,$ $\newdim=\originaldim,$ $\{\matrix\in\R^{p\times p}:\matrix\text{~symmetric}\},$ $\Tdatap_{ij}(\datav)=\sqrt{\dataa_i\dataa_j},$ and $\htermdatav=-\sum_{j=1}^p\log(\dataa_i!)\,$. The relationships among the fish populations are then encoded in~\matrix, which then needs to be estimated from the observations.

Before heading on, we add some convenient notation. We summarize the data in $\datatotal:=(\datasseq)$ and denote the corresponding function argument by $\datatotala:=(\datasseqa)$ for $\datasseqa\in\Domain.$ The generalized Gram matrix is denoted by
\begin{equation*}
 \Tdatae(\datatotala):=\sino \Tdatap(\datav^i)\,. 
\end{equation*}
The negative joint log-likelihood function $-\ell_\tensor$ for $n$ i.i.d random vectors corresponding to the model~\eqref{model} is finally given by
  \begin{equation*}
-\ell_\tensor(\datatotala)= n\inprodtr{\tensor}{\Tdatae(\datatotala)}-\si\hterm(\datav^i)+n\normtotal(\tensor)\, .
\end{equation*}

We can now state the maximum likelihood estimator and its properties. For further reference, we first state the essence of the previous discussion in the following lemma.
\begin{lemma}[Log-likelihood]\label{loglike}
Given any $\tensor\in\Matrixint$,  the negative joint log-likelihood function~$-\ell_\tensor$  of $n$ i.i.d.\@ random vectors distributed according to~$f_\matrix$ in~\eqref{model} can be expressed by
  \begin{equation*}
-\ell_\tensor(\datatotala)= n\inprodtr{\tensor}{\Tdatae(\datatotala)}+n\normtotal(\tensor)+c\, ,
\end{equation*}
where $c\in\R$ does not depend on $\tensor.$
\end{lemma}
\noindent Motivated by Lemma~\ref{loglike}, we  introduce the maximum likelihood estimator of \tensortrue\ by
\begin{equation}\label{def:mle}
  \tensorest:=\argmin_{\matrixtoy\in\Tensor}\big\{-\ell_\matrixtoy(\datatotal)\big\}=\argmin_{\matrixtoy\in\Tensor}\big\{\inprodtr{\matrixtoy}{\Tdatae(\datatotal)}+\normtotal(\matrixtoy)\big\}\,.
\end{equation}
The estimator exists in all generic examples. More generally, under our assumption  that  $\matrix\in\Matrix\subset\Matrixint$ and $\Matrix$ is open and convex, and the minimizer exists for  $n$ sufficiently large, cf.~\citep{Berk72}. In particular, the objective function is convex, and its derivatives can be computed explicitly.
\begin{lemma}[Convexity]\label{convexity}
For any $\datatotala\in\Domain^p,$ the function
  \begin{align*}
\Matrixint\to&\ \R\\
\tensor\mapsto&\ \inprodtr{\tensor}{\Tdatae(\datatotala)}+\normtotal(\tensor)
\end{align*}  
is convex.
\end{lemma}
\begin{lemma}[Derivatives]\label{derivative}
For any $\datatotala\in\Domain^p,$ the function
  \begin{align*}
\Matrixint\to&\ \R\\
\tensor\mapsto&\ \inprodtr{\tensor}{\Tdatae(\datatotala)}+\normtotal(\tensor)
\end{align*}  
 is twice differentiable with partial derivatives
 \begin{equation*}
\frac{\partial}{\partial \tensor{\tensorindex}}\left(\inprodtr{\tensor}{\Tdatae(\datatotala)}+\normtotal(\tensor)\right)=\Tdatae{\tensorindex}(\datatotala)-\E_\tensor\Tdatae{\tensorindex}(\datatotal)
      \end{equation*}
and
\begin{align*}
&\frac{\partial}{\partial \tensor\tensorindex}\frac{\partial}{\partial \tensor\tensorindexp}\left(\inprodtr{\tensor}{\Tdatae(\datatotala)}+\normtotal(\tensor)\right)\\
&~~~~~~~=n\,\E_\tensor\left[\big(\Tdatae\tensorindex(\datatotal)-\E_\tensor\Tdatae\tensorindex(\datatotal)\big)\big(\Tdatae\tensorindexp(\datatotal)-\E_\tensor\Tdatae\tensorindexp(\datatotal)\big)\right]
\end{align*}
for $i,j,k,l\in\otq.$
\end{lemma}
\noindent \noindent Convexity and the explicit derivatives are desirable for both optimization and theory. From an optimization perspective, the two properties are valuable, because they render the objective function amenable to gradient-type minimization. From a theoretical perspective, the two properties are valuable, because they  imply that
\begin{equation*}
  \matrix=\argmin_{\matrixtoy\in\Tensor}\big\{\E_{\matrix}\big[\inprodtr{\matrixtoy}{\Tdatae(\datatotal)}+\normtotal(\matrixtoy)\big]\big\}\,,
\end{equation*}
showing that \tensorest\ is a standard M-estimator, and because they imply that \tensorest\ can be written as a Z-estimator (note that \tensorest\ is necessarily in the interior of \Tensor) with criterion
\begin{equation*}
\Tdatae(\datatotal)=\E_\tensorest\Tdatae(\datatotal)\,.
\end{equation*}

A simple special case is the multivariate Gaussian model.  Recall that in this case, \matrixtrue\ is the inverse of the usual population covariance matrix. Moreover, one can check that
\begin{equation*}
\Tdatae(\datatotala)=\sino\datav^i \datav^i{}\tp
\end{equation*}
and $\matrixest=\Tdatae(\datatotal)^{-1}.$ Hence, in this case, the estimator \matrixest\ is the inverse of the (usual) sample covariance matrix.

\begin{remark}
	 The maximum likelihood estimator of~\matrix\ is asymptotically normal with covariance equal to the inverse Fisher information. In particular,  consistency and asymptotic normality of the maximum likelihood estimator can be proved following the arguments in the classical paper~\citep{Berk72}, see especially~\cite[Theorems~4.1 and 6.1]{Berk72}. We refer to that paper for details.
\end{remark}

\subsection{Algorithm}\label{sec:algorithm}
The main challenge in computing the maximum likelihood estimator is the unconventional normalization term. 
We address this challenge by approximating the objective function using a sampling-based technique. 
In this section, we describe the corresponding algorithm.

We denote the objective function of the maximum likelihood estimator in Equation~\eqref{def:mle} as
\begin{equation*}
	g(\matrixtoy) := \inprodtr{\matrixtoy}{\Tdatae(\datatotal)}+\normtotal(\matrixtoy).
\end{equation*}
The normalization term $\normtotal(\matrixtoy)$, in general, does not have a closed-form formula and, therefore, makes the objective function hard to compute exactly. 
However, we show in the following that it can be feasibly approximated.
Adding the constant term $\normtotal(\tensor_0)$, where $\tensor_0$ is a pre-specified constant parameter matrix,  to the objective function of~\eqref{def:mle} yields an equivalent definition of the maximum likelihood estimator as
\begin{equation}
\tensorest =\argmin_{\matrixtoy\in\Tensor}\big\{\inprodtr{\matrixtoy}{\Tdatae(\datatotal)}+\normtotal(\matrixtoy) - \normtotal(\tensor_0)\big\}\,.
\end{equation}
Algebraic transformation reveals 
\begin{align*}
\normtotal(\matrixtoy)-\normtotal(\tensor_0) = \log \E_{\tensor_0}\Bigl( e^{-\inprodtr{\matrixtoy-\tensor_0}{T(\datav)}} \Bigr).
\end{align*} 
The finite expectation can be approximated by an empirical mean based on some sample set $\mcset$ from the distribution $f_{\tensor_0}(\datav):=\exp(-\inprodtr{\tensor_0}{\Tdatap(\datav)}+\hterm(\datav)-\normtotal(\tensor_0))$. 
By the strong law of large numbers, when the cardinality of the set $\mcset$  (denoted as $|\mcset|$) goes to infinity, the empirical mean approximates the expectation well:
	\begin{equation*}
	\log \frac{1}{|\mcset|}\sum_{\datamc \in \mcset}e^{-\inprodtr{\tensor-\tensor_0}{T(\datamc)}}  \xrightarrow[]{a.s}
	\log \E_{\tensor_0}\Bigr( e^{-\inprodtr{\tensor-\tensor_0}{T(\datav)}} \Bigl).
	\end{equation*}
Hence, in practice, we solve the approximate problem
\begin{equation}\label{eq: approximate_MLE}
\tensorest\approx
\argmin_{\matrixtoy\in\Tensor}\big\{\widetilde{g}(\matrixtoy) \big\}\,,
\end{equation}
where
\begin{equation*}
	\widetilde{g}(\matrixtoy) := \inprodtr{\matrixtoy}{\Tdatae(\datatotal)}+ \log \frac{1}{|\mcset|}\sum_{\datamc \in \mcset}e^{-\inprodtr{\matrixtoy-\tensor_0}{\Tdatap(\datamc)}}\,.
\end{equation*}
We apply gradient descent to solve the problem of~\eqref{eq: approximate_MLE}.  
The gradient with respect to $\matrixtoy$ is
\begin{equation}\label{approx gradient}
\nabla \widetilde{g}(\matrixtoy)= \Tdatae(\datatotal) -  \frac{\sum_{\datamc \in \mcset} \Tdatap(\datamc) e^{-\inprodtr{\matrixtoy-\tensor_0}{\Tdatap(\datamc)}}}{ \sum_{\datamc \in \mcset} e^{-\inprodtr{\matrixtoy-\tensor_0}{\Tdatap(\datamc)}}} .
\end{equation}
\begin{remark}
	The gradient in~\eqref{approx gradient} can also be considered as an instantiation of self-normalized importance sampling~\citep{owen2016monte}.
	Lemma~\ref{derivative} gives the gradient of $g(\matrixtoy)$ in the form of
	\begin{equation*}
		\nabla g(\matrixtoy) = \Tdatae(\datatotal)-\E_\matrixtoy\Tdatae(\datatotal) = \Tdatae(\datatotal)-\E_\matrixtoy\Tdatap(\data) \,.
	\end{equation*}
	When $\E_\matrixtoy\Tdatap(\data)$ lacks an algebraic expression and $f_{\matrixtoy}$ is inconvenient to sample from, importance sampling~\citep{owen2016monte} is useful to approximate the expectation term $\E_\matrixtoy\Tdatap(\data)$. The idea of importance sampling is to draw samples from a biased distribution and obtain the desired $f_{\matrixtoy}$ by adjusting the weights for the drawn samples. Self-normalized importance sampling refers to the special case when weights are normalized by their sum. 
Equation~\eqref{approx gradient} reflects the same idea and uses the pre-specified $f_{\tensor_0}$ as the biased sampling distribution. 
\end{remark}

The choice of $\tensor_0$ is essential for the finite-sample performance of the approximation. 
Our two main considerations are: 
First, the sampling distribution~$f_{\tensor_0}$ should be straightforward for generating samples. 
Second, it should lead to balanced weights and a small variance of $\widehat{\E}_\matrixtoy \Tdatap(\data)$;
if weights concentrate in just a few samples, we have effectively only got these observations, resulting in large variability of the approximation~\citep{owen2016monte}. 
To incorporate the two considerations, we propose
\begin{equation}
	\label{eq: M_0 optimization}
	\tensor_0 :=  \argmin_{\matrixtoy\in\Tensor}\big\{\inprodtr{\matrixtoy}{\Tdatae(\datatotal)}+\normtotal(\matrixtoy) \big\}~~ \mbox{subject to}~~ \matrixtoy_{kl} = 0,~~ \forall k\neq l.
\end{equation}
We restrict the parameter to a diagonal matrix, by which we presume mutual independence among all coordinates and disentangle the objective function. 
Hence, both the optimizing problem~\eqref{eq: M_0 optimization} and the task of sampling from $f_{\tensor_0}$ can be handled for each coordinate separately, and each coordinate reduces to a standard univariate exponential family distribution. Besides, $\tensor_0$ is the diagonal matrix closest to the actual parameter. When the off-diagonal entries of the actual parameter matrix are small, weights of the drawn samples are expected to be reasonable.

Algorithm~\ref{algorithm} summarizes our computational pipeline. 
In the full version, which is stated in Appendix~\ref{sec: full algorithm}, a backtracking line search selects the step size adaptively and incorporates the domain constraint as needed (for example, the positive definite requirement for the case of continuous measures).

\begin{algorithm}
\caption{Solving for the maximum likelihood estimator}
\label{algorithm}
\SetKwInOut{Input}{Input}
\SetKwInOut{Output}{Output}
\tcp{$\eta$: step size}
\Input{$\Tdatae(\datatotal),\eta > 0$}
\Output{$\tensorest$}
\tcp{Solve for $\tensor_0$}
	
	$\tensor_0 \leftarrow \mathbf{0}_{p\times p}$\;
	
	\For{$i = 1, \dots, p$}{
		$(\tensor_0)_{ii} \leftarrow \argmin_{m\in\R}\Big\{m\Tdatae_{ii}(\datatotal)+  \log \int \exp \big(-m\Tdatap_{ii}(\datav) + \xi(\dataa_i)\big)d\dataa_i \Big\}$\;
	}
\tcp{Generate sample set $\mcset$ from $f_{\tensor_0}=\prod_{i=1}^p f_{(\tensor_0)_{ii}}(\dataa_i)$}
	\For{$i = 1, \dots, p$}{
	Generate 10,000 random samples from $f_{(\tensor_0)_{ii}}(\dataa_i)$ for the $i$-th coordinate\;
}
\tcp{Apply gradient descent with constant step size to~\eqref{eq: approximate_MLE}}
$k\leftarrow0$\;
$\matrixtoy_k\leftarrow\tensor_0$\;
\Repeat{$\big|\widetilde{g}({\matrixtoy_{k}}) - \widetilde{g}({\matrixtoy_{k-1}})\big| < 10^{-4}$}{
	$k \leftarrow k+1$\;
	$\nabla \widetilde{g}(\matrixtoy_{k-1}) \leftarrow \Tdatae(\datatotal) -  \frac{\sum_{\datamc \in \mcset} \Tdatap(\datamc) e^{-\inprodtr{\matrixtoy_{k-1}-\tensor_0}{\Tdatap(\datamc)}}}{ \sum_{\datamc \in \mcset} e^{-\inprodtr{\matrixtoy_{k-1}-\tensor_0}{\Tdatap(\datamc)}}}$\;
	$\matrixtoy_{k} \leftarrow \matrixtoy_{k-1} - \eta\nabla \widetilde{g}(\matrixtoy_{k-1})$\;
}
$\tensorest \leftarrow \matrixtoy_{k}$\;
\end{algorithm}

\section{Simulation Studies}\label{sec:simulation}
Exponential trace models apply to a large variety of multivariate data that have correlated coordinates. 
The framework is especially useful for data that are discrete, heavy-tailed, or composed of different data types. 
We consider the following four model-types in simulations: Poisson, Exponential, Poisson-Bernoulli (a composite of Poisson and Bernoulli), and Poisson-Gaussian (a composite of Poisson and Gaussian).
Such types of models are useful in practice but, to date, have proven challenging to characterize and estimate.

\subsection{Settings}\label{sec: sim setting}
We follow the discussion of Section~\ref{sec:examplesnew} and consider the square-root transformation for non-Gaussian data as one example,
that is, we set $\Tdatap_{ij}(\datav) = t_i(\dataa_i)t_j(\dataa_j),$ where $t_i(\dataa_i) = \sqrt{\dataa_i}$ for non-Gaussian coordinates and  $t_i(\dataa_i) = \dataa_i$ for Gaussian coordinates.

In the following, we describe the specific graph structures for discrete data, continuous data, and composite data of both types. 
The discrete data category also covers composite data of different discrete types (for example, Poisson-Bernoulli).

\subsection*{Discrete Data}
We consider \er\  random graphs and generate the corresponding $p \times p$ parameter matrix $\tensor$ in the following manner. Let $c_{0}$ and $c_{1}$ be two constants ($c_1 \neq 0$). We set the diagonal entries to $\tensor_{ii} = c_0.$
For $i \neq j$, the off-diagonal entries are independent and identically distributed as
\begin{equation}\label{eq: discrete Mij}
	\tensor_{ij} = \tensor_{ji} = 
	\begin{cases}
		\phantom{-}c_1 & \text{with probability } 1/p,\\
		-c_1 & \text{with probability } 1/p,\\
		\phantom{-}0 & \text{with probability } 1-2/p.
	\end{cases}   
\end{equation}
The corresponding \er\  random graph has $p-1$ edges in expectation, among which half represent positive interactions and the other half  negative ones. 

\subsection*{Continuous Data}
We consider again \er\ random graphs. In addition, we generate strictly diagonally dominant matrices with positive diagonal entries for the parameter matrix $\tensor$.  
It can be shown that the $\tensor$'s are positive definite and satisfy the integrability condition. More specifically, we first generate a $p \times p$ adjacency matrix with i.i.d.\@ off-diagonal entries
\begin{equation*}
	A_{ij} = A_{ji} = 
	\begin{cases}
		\phantom{-}1 & \text{with probability } 1/p,\\
		-1 & \text{with probability } 1/p,\\
		\phantom{-}0 & \text{with probability } 1-2/p,
	\end{cases}   
\end{equation*}
We denote the maximum node degree by $s := \max_{i} \sum_{j\neq i}|A_{ij}|.$ 
Then, a positive definite $\tensor$ can be generated by setting the diagonal entries to 1 and 
the off-diagonal entries to
\begin{equation*}
	\tensor_{ij} = \tensor_{ji} = 1/(s+0.1) A_{ij}.
\end{equation*}
The corresponding \er\ random graph has $p-1$ edges in expectation, among which half represent positive interactions and the other half  negative ones. 

\subsection*{A Composite of Discrete and Continuous Data}
We consider even values of $p$, with the first $p_1 = p/2$ coordinates discrete, and the remaining $p_2 = p/2$ coordinates continuous.
The parameter matrix in blockwise format is
\begin{equation*}
	\tensor=
	\left[
	\begin{array}{cc}
		\tensor_{11} & \tensor_{12} \\
		\tensor_{12}^\top & \tensor_{22}
	\end{array}
	\right],
\end{equation*}
where $\tensor_{11}, \tensor_{22}$ represent the conditional dependences among $p_1$ discrete and $p_2$ continuous coordinates, respectively. 
The remaining block $\tensor_{12}$ describes the conditional dependences between discrete and continuous coordinates. When the continuous node conditionals follow a Gaussian distribution, the integrability condition is satisfied for all $\tensor \in \R^{p \times p}$ such that $\tensor_{22}$ is positive definite. Therefore, we generate $\tensor_{22}$ in the above described continuous case to guarantee its positive definiteness while generating $\tensor_{11}$ and $\tensor_{12}$ in the above described discrete case.

\subsection{Results}
We evaluate the maximum likelihood estimators in terms of edge recovery by studying average ROC (receiver operating characteristic) curves based on thresholding maximum likelihood estimates. 
Each average ROC curve is taken over 50~individual ROC curves that correspond to 50~different \er\  random graphs. ROC curves are combined using horizontal-averaging via the \texttt{R} package \texttt{ROCR}~\citep{sing2005rocr}. Since Gaussian graphical models are currently widely used, even in cases with obvious misspecification (eg. count data)~\citep{zhao2019cancer},
we compare the maximum likelihood estimators of the exponential trace model to that of the (misspecified) Gaussian graphical model.

The graph structures are described in Section~\ref{sec: sim setting}. Details regarding the data generation approaches are deferred to Appendix~\ref{sec:data generation}. 
We use $n=250$ independent observations to recover the conditional dependence of $p=20$ variables.
The diagonal entry is $c_0 = -1$ and the off-diagonal entry is $c_1 = 0.3$. 
We show the average ROC curves of Poisson, Exponential, Poisson-Bernoulli, and Poisson-Gaussian in Figures~\ref{fig:pure_data_ROC} and~\ref{fig:composite_data_ROC}. In addition, we explore scenarios with diagonal entry $c_0 \in\{0,-0.5,  -1,  -1.5,  -2\}$ and off-diagonal entry $c_1 \in\{ 0.1,  0.2,  0.3,  0.4,  0.5\}:$ we show the differences in AUC of average ROC curves between the exponential trace model and the Gaussian graphical model in Figures~\ref{fig:pure_data_AUC} and~\ref{fig:composite_data_AUC}.
\begin{figure}
	\centering
	\includegraphics[width=\textwidth]{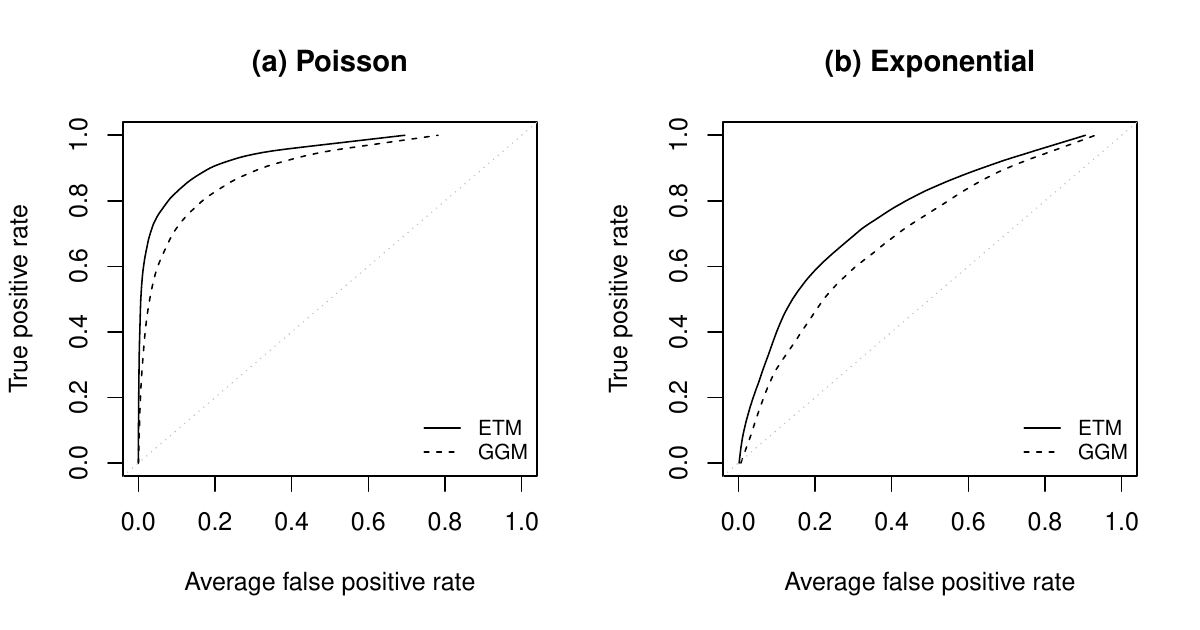} 
	\vskip-3ex
	\caption{Average ROC curves for pure data. ETM stands for the Exponential trace mode and GGM stands for Gaussian graphical model.}
	\label{fig:pure_data_ROC}
\end{figure}

\begin{figure}
	\centering
	\includegraphics[width=\textwidth]{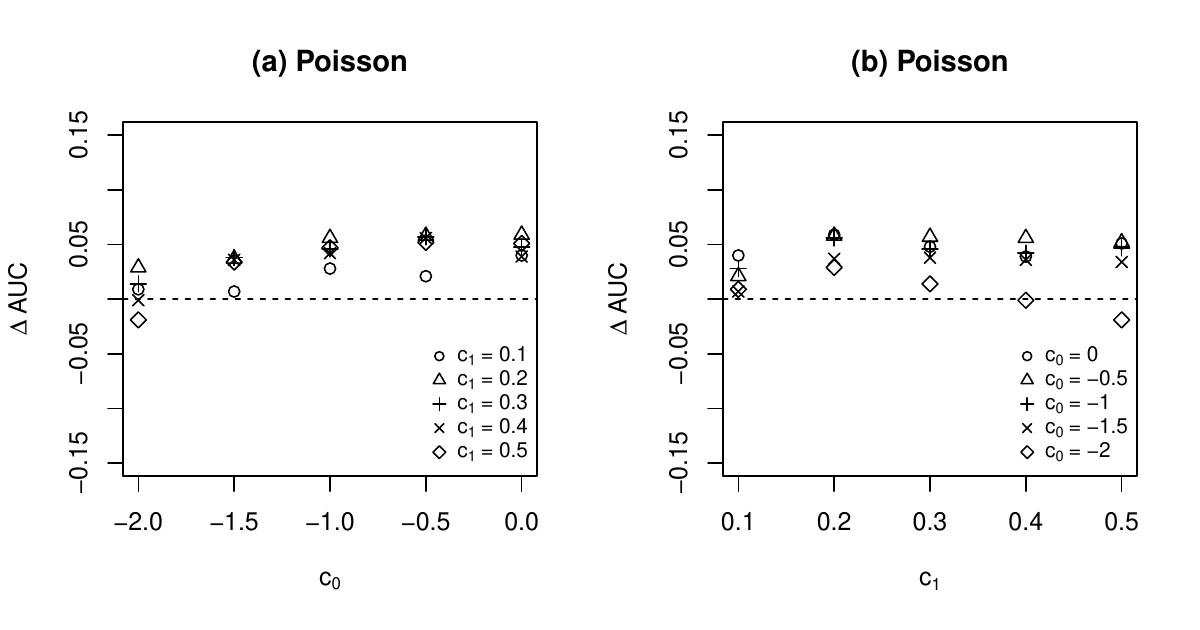} 
	\vskip-3ex
	\caption{Differences in AUC of average ROC curve between exponential trace model and Gaussian graphical model for Poisson data.}
	\label{fig:pure_data_AUC}
\end{figure}

\begin{figure}
	\centering
	\includegraphics[width=\textwidth]{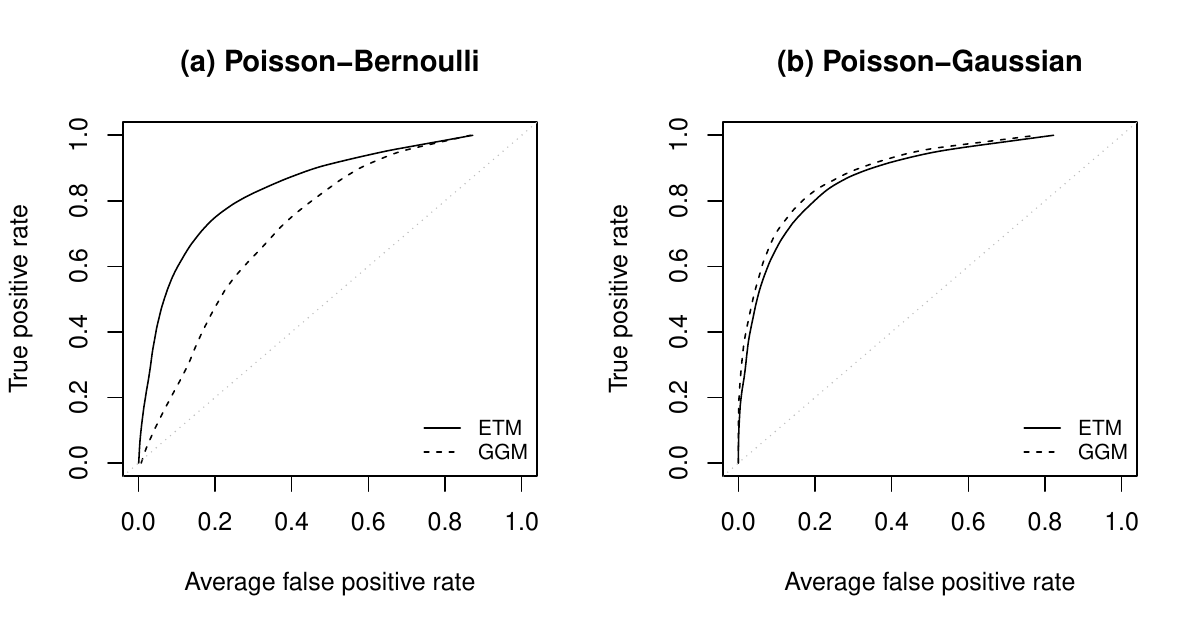}
	\vskip-3ex
	\caption{
		Average ROC curves for composite data. ETM stands for the Exponential trace mode and GGM stands for Gaussian graphical model.}
	\label{fig:composite_data_ROC}
\end{figure}

\begin{figure}
	\centering
	\includegraphics[width=\textwidth]{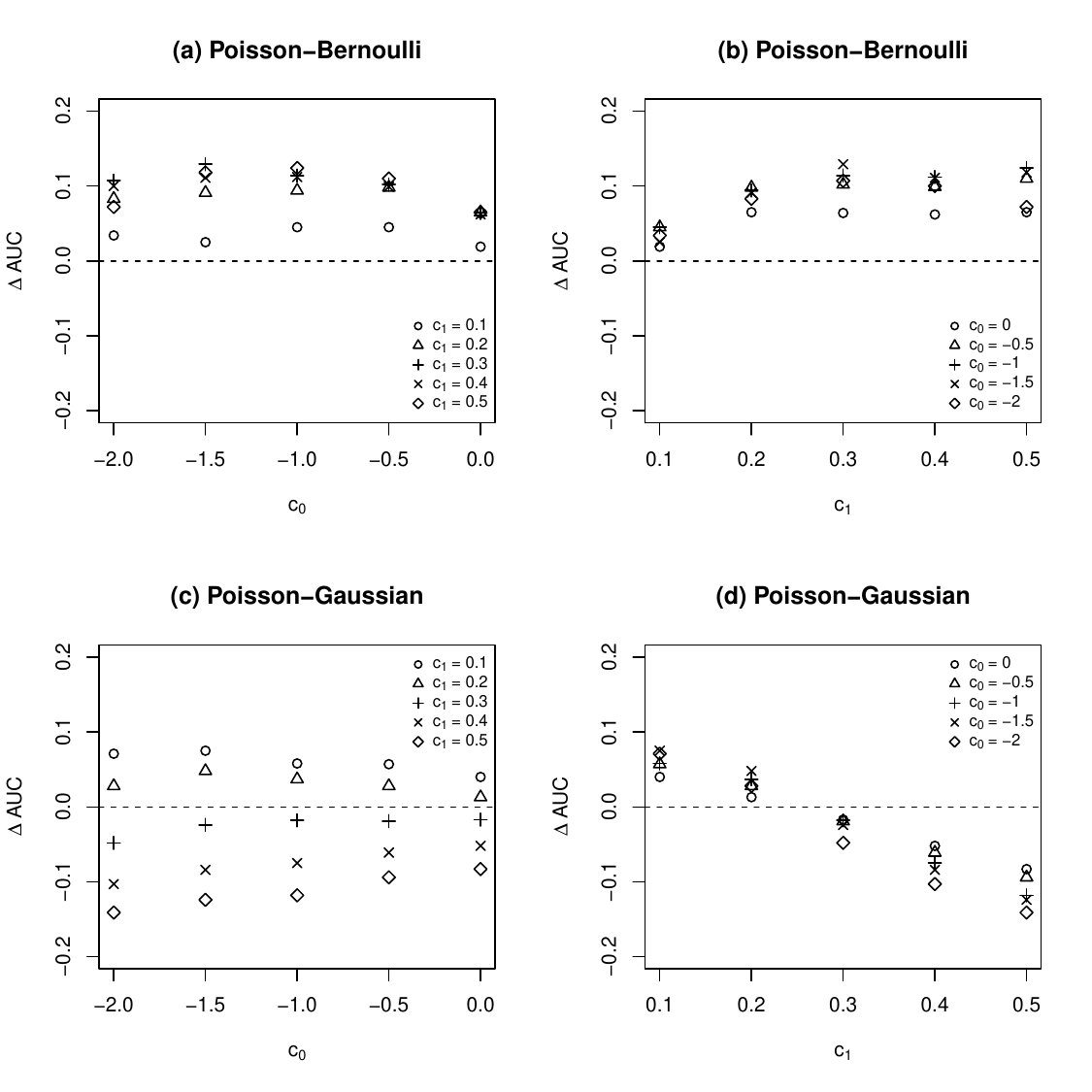}
	\vskip-3ex
	\caption{Differences in AUC of average ROC curves between exponential trace model and Gaussian graphical model for composite data.}
	\label{fig:composite_data_AUC}
\end{figure}

In the pure data type scenarios:  The exponential trace model outperforms the Gaussian graphical model substantially---the expoential trace model's  ROC curves lie entirely above the Gaussian graphical model's ROC curves---in the scenario of small interaction and small sufficient statistics $\Tdatap(\datav)$ (see Figure~\ref{fig:pure_data_ROC}).
In addition, we look into more configurations of Poisson type scenarios by varying $c_0$ and $c_1$:
the performance of the sampling-based approximation determines that of the exponential trace model (see Figure~\ref{fig:pure_data_AUC}).
When the interaction term and sufficient statistics are small to moderate, the exponential trace model has a larger AUC than the Gaussian graphical model. 
But when the interaction term and sufficient statistics are large, for example, $c_0 = -2$ and  $c_1 = 0.5$, the improvement is not guaranteed. 
For the composite data type scenarios:  The exponential trace model shows improved performance for Poisson-Bernoulli data but not for Poisson-Gaussian data (see Figures~\ref{fig:composite_data_ROC} and~\ref{fig:composite_data_AUC}). 
Bernoulli data have very small sufficient statistics so that the improvement in Poisson-Bernoulli data is substantial for a large range of interaction terms,
but Poisson-Gaussian data involves Gaussian coordinates, so that the performance of the exponential trace model is mixed. 
The comparative performance is not only determined by the performance of the approximation algorithm but also the degree to which the conditional dependence structure resembles the zero pattern of the precision matrix. 

In conclusion, the simulation study shows that: (1)~the exponential trace model can improve on the Gaussian graphical model for non-Gaussian data, especially when the sufficient statistics and small interactions are small; 
(2)~the approximation approach can struggle when sufficient statistics and interactions are large.
Based on this, we believe that our modeling approach has substantial potential: Some of this potential is realized by our current implementation, however some of the potential might require additional computational insights.


\section{Application to Neural Spike Data}
\label{sec:application}
In this section, we apply the proposed exponential trace model to neural spike data.
The temporal and spatial patterns of neural spikes capture the concurrent activity of neurons. Understand this is essential for learning neural circuits. Neural spike data is usually formulated as spike counts in a short time bin and modeled by a Poisson distribution~\citep{theis2016benchmarking}. We consider a data set of multi-electrode array recordings of spike trains in mouse retina~\citep{Demas:2003fq} obtained from the Retinal Wave Repository~\citep{waveRepository}.
We transform the spike time data into spike counts in time bins of $40\,ms$ following conventions in neural science~\citep{theis2016benchmarking}. 
The short time interval captures the instantaneous characteristics of neuron firing. 
The recording covers an $800 \times 800\,\mu m$ surface area and provides locations of each recorded unit in the form of $(x, y)$-coordinates. The number of recorded units ranges from 12 to 22 in different mice.  

The spike counts of each recorded unit range from 0 to 13 and roughly follow the mean-variance relationship of a Poisson distribution. The small counts imply: (1) the exponential trace model with square-root transformation is close to the Poisson one; (2) the properties of the data fit the scenario for which the proposed algorithm is appropriate.
These two observations render the exponential trace model with a square-root transformation appropriate.

In Figures~\ref{fig: 6WK_WT}, we present the recovered connections of recorded units for a 6-week-old wild-type mouse. 
To obtain sparse graphs, the maximum likelihood estimator is thresholded to 30~edges.
In the plotted graphs, the positions of the nodes correspond to to [slightly jittered] $(x,y)$-coordinates of the neurons in the recording. This allows us to consider spatial summaries, eg. the average physical length of edges in our graph.


The exponential trace model finds a more centralized graph than the Gaussian graphical model, that is,  neurons located together tend to connect in the exponential trace model approach while the estimated connections in the graph are longer in the Gaussian approach.
To evaluate this difference quantitatively, we compute the mean Euclidean distance between all pairs of directly connected neurons.
In the 6-week-old wild-type mouse, the mean distance is $169\,\mu m$ (SE: $24\,\mu m$) under the exponential trace model and $252\,\mu m$  (SE: $36\,\mu m$) under the Gaussian graphical model.
In addition, the exponential trace model recovers a main connection component and an isolated neuron unit (see Unit 7 in Figure~\ref{fig: 6WK_WT}). 
The isolated neuron unit is physically far away from the primary component and may belong to another functional group. In contrast, the Gaussian graphical model does not distinguish the location separation clearly but instead connects almost every units. 

These characteristics found in the exponential trace model but not the Gaussian graphical model align with our biological understanding. 
In particular, neurons transmit signals to others through synapses, which are physical connections between two neurons~\citep{lodish2008molecular}. This biological mechanism favors direct coordination of closely located neurons. Specifically, previous studies find that the degree two units spike together within some small time window decays with the distance separating the neurons in retina~\citep{Cutts:2014cw, wong1993transient,xu2011instructive}.

\begin{figure}
	\centering
	\includegraphics[width=\textwidth]{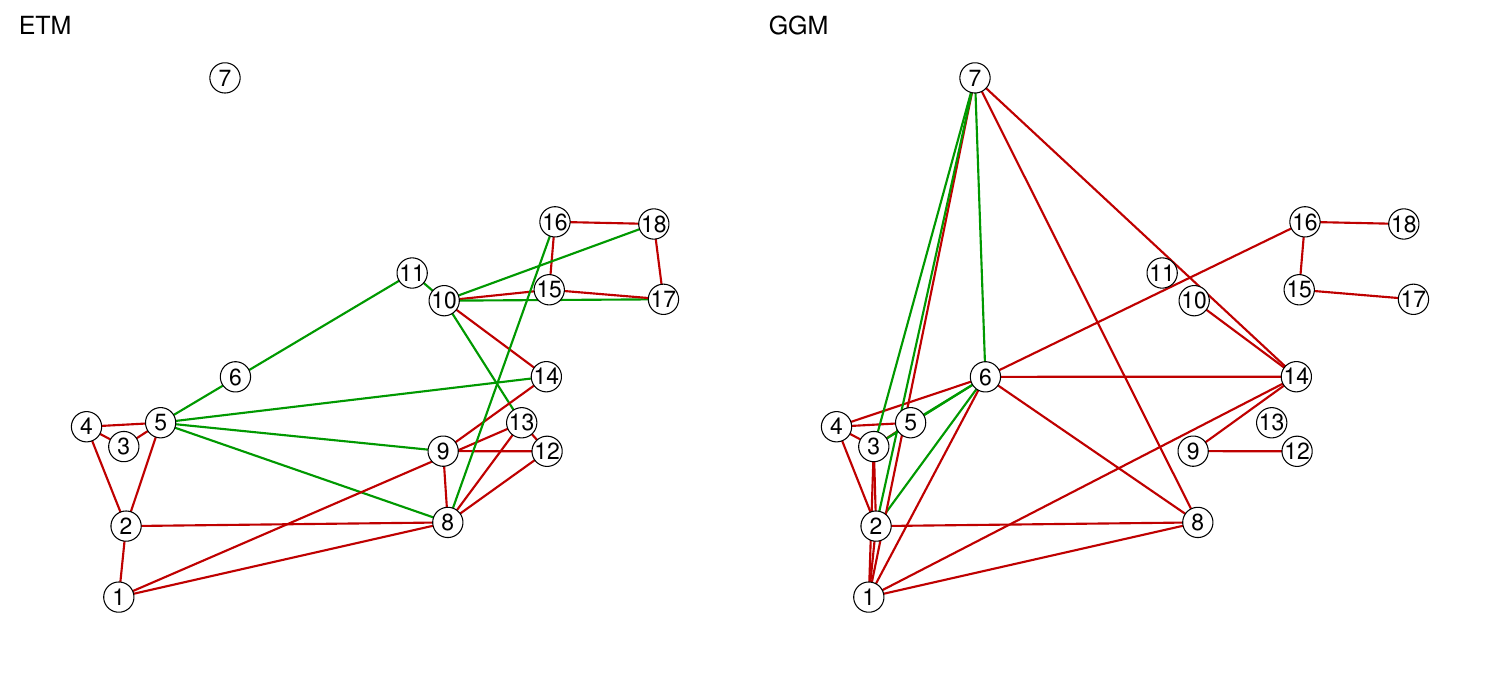} 
	\vskip-3ex
	\caption{Retinal connections recovered for a 6-week-old wild-type mouse of the Demas et al.\@ 2003 data. 
The positive interactions are drawn in red,and the negative ones in green. The left panel displays the connection recovered by the exponential model with square-root transformation; the right panel displays the connection recovered by the Gaussian graphical model.}
	\label{fig: 6WK_WT}
\end{figure}


\section{Discussion}\label{sec:discussion}
In this paper we proposed an exponential family-based framework to build graphical models.
This framework for graphical models allows for a wide range of data types as highlighted in Sections~\ref{sec:examplesstandard} and~\ref{sec:examplesnew} and yet ensures a rigid theoretical structure as demonstrated in Section~\ref{sec:estimation}.
The models are amenable to estimation based on maximum likelihood, which can be computed through a sampling-based approximation technique.

\section*{Acknowledgements} 
We also thank Jon Wellner, Marco Rossini, Mathias Drton,  Sam Koelle, Zack Almquist, Ali Shojaie, Lina Lin for their valuable input.


\bibliographystyle{agsm}
\bibliography{References}


\pagebreak
\appendix
\section{Proofs}

\subsection{Proof of Lemma~\ref{propertiesfamily}}

\begin{proof}[Proof of Lemma~\ref{propertiesfamily}]
  We prove the two properties in order. The main proof ideas can also be found in~\cite[Pages~193-195]{Berk72}.
\paragraph{Property 1} We first show that $\Matrixint$ is convex. 

For this, consider $\alpha\in[0,1]$ and $\matrix,\matrix'\in\Matrixint.$ Then, by definition of the normalization and by convexity of the exponential function,
\begin{align*}
e^{\normtotal(\alpha\matrix+(1-\alpha)\matrix')}
=&\int_\Domain e^{-\inprodtr{\alpha\matrix+(1-\alpha)\matrix'}{\Tdatap(\datav)}+\hterm(\datav)}\,d\measure\\
 \leq& \int_\Domain \big(\alpha e^{-\inprodtr{\matrix}{\Tdatap(\datav)}+\hterm(\datav)}+(1-\alpha)e^{-\inprodtr{\matrix'}{\Tdatap(\datav)}+\hterm(\datav)}\big)\,d\measure\\
 =&\,\alpha e^{\normtotal(\matrix)}+(1-\alpha)e^{\normtotal(\matrix')}
<\infty\,.
\end{align*}
Hence, $\normtotal(\alpha\matrix+(1-\alpha)\matrix')<\infty$, and thus, $\alpha\matrix+(1-\alpha)\matrix'\in\Matrixint.$ This concludes the proof of the first property.
\paragraph{Property 2} We now show that  for any $\matrix\in\Matrixint$, the coordinates of~$\Tdatap(\data)$ have moments of all orders with respect to $f_\matrix$. 

To this end, fix an $\matrix\in\Matrixint$. Since \Matrixint\ is open, there is a neighborhood ${\mathfrak M}_\matrix$ of $0_{\newdim\times \newdim}$ such that $\{\matrix-A:A\in{\mathfrak M}_\matrix\}\subset\Matrixint$.  For any $A\in{\mathfrak M}_\matrix$,  the moment generating function of $\Tdatap(\data)$ is finite:
\begin{align*}
  \E_\matrix e^{\inprodtr{A}{\Tdatap(\data)}}=\int_\Domain e^{-\inprodtr{\tensor-A}{\Tdatap(\datav)}+\hterm\left(\datav\right)-\normtotal(\matrix)}\,d\measure=e^{\normtotal(\matrix-A)-\normtotal(\matrix)}<\infty\,.
\end{align*}
This is a sufficient condition for the existence of all moments of~$\Tdatap(\data)$~\cite[Page~33]{Shao03} and thus concludes the proof of the second property.
\end{proof}

\subsection{Proof of Lemma~\ref{convexity}}

\begin{proof}[Proof of Lemma~\ref{convexity}] The claim follows readily from Lemma~\ref{derivative}, which is proved in the next section. Indeed, using the second derivates stated in Lemma~\ref{derivative}, we find for any $\tensor\in\Matrixint$ and $\tensor'\in\tensorspace,$
\begin{align*}
 &  \sum_{i,j,k,l=1}^\newdim\tensor\tensorindex' \,\frac{\partial}{\partial \tensor\tensorindex}\frac{\partial}{\partial \tensor\tensorindexp}\left(\inprodtr{\tensor}{\Tdatae(\datatotala)}+\normtotal(\tensor)\right)\tensor\tensorindexp'\\
=&\, n \sum_{i,j,k,l=1}^\newdim\tensor\tensorindex'\, \E_\tensor\left[\big(\Tdatae\tensorindex(\datatotal)-\E_\tensor\Tdatae\tensorindex(\datatotal)\big)\big(\Tdatae\tensorindexp(\datatotal)-\E_\tensor\Tdatae\tensorindexp(\datatotal)\big)\right]\tensor\tensorindexp'\\
=&\, n\,\E_\tensor\inprodtr{\tensor'}{\Tdatae(\datatotal)-\E_\tensor\Tdatae(\datatotal)}^2\,.
\end{align*}
The display implies that for any $\tensor\in\Matrixint$ and $\tensor'\in\tensorspace,$
\begin{align*}
  \sum_{i,j,k,l=1}^\newdim\tensor\tensorindex' \,\frac{\partial}{\partial \tensor\tensorindex}\frac{\partial}{\partial \tensor\tensorindexp}\left(\inprodtr{\tensor}{\Tdatae(\datatotala)}+\normtotal(\tensor)\right)\tensor\tensorindexp'\ \geq\ 0\,.
\end{align*}
This ensures convexity, and thus concludes the proof of Lemma~\ref{convexity}.
\end{proof}

\subsection{Proof of Lemma~\ref{derivative}}

\begin{proof}[Proof of Lemma~\ref{derivative}]
We prove the two claims in order.
\paragraph{Part 1} We start by taking the first derivative, showing that 
\begin{equation*}
  \nabla\tensorindex\left(\inprodtr{\tensor}{\Tdatae(\datatotala)}+\normtotal(\tensor)\right)=\Tdatae{\tensorindex}(\datatotala)-\E_\tensor\Tdatae{\tensorindex}(\datatotal)\,,
\end{equation*}
where we use the shorthand notation $\nabla\tensorindex:=\frac{\partial}{\partial \tensor\tensorindex}\,.$

Since the trace is linear, the derivative of the first term is
\begin{align*}
  \nabla\tensorindex\inprodtr{\tensor}{\Tdatae(\datatotala)}=\Tdatae\tensorindex(\datatotala)\,.
\end{align*}
For the second term, recall  that the normalization~\normtotal\ is given by
\begin{equation*}
\normtotal(\tensor)=\linkinv\int_\Domain e^{-\inprodtr{\tensor}{\Tdatap(\datav)}+\hterm(\datav)}d\measure\,.
\end{equation*}
Taking exponentials on both sides, we find
\begin{equation*}
\link^{\normtotal(\tensor)}=\int_\Domain\link^{-\inprodtr{\tensor}{\Tdatap(\datav)}+\hterm(\datav)}d\measure\,.
\end{equation*}
We can now take derivatives and get 
\begin{align*}
\linkd^{\normtotal(\tensor)}  \nabla\tensorindex\normtotal(\tensor)=&\int_\Domain\nabla\tensorindex\link^{-\inprodtr{\tensor}{\Tdatap(\datav)}+\hterm(\datav)}d\measure\\
=&-\int_\Domain\Tdatap\tensorindex(\datav)\linkd^{-\inprodtr{\tensor}{\Tdatap(\datav)}+\hterm(\datav)}d\measure\,,
\end{align*}
where we again use the linearity of the trace. Bringing the exponential factor back into the integral and using the assumed independence of the observations then yields
\begin{align*}
 \nabla\tensorindex\normtotal(\tensor)=&-\int_\Domain\Tdatap\tensorindex(\datav)\linkd^{-\inprodtr{\tensor}{\Tdatap(\datav)}+\hterm(\datav)-\normtotal(\tensor)}d\measure=-\E_\tensor\Tdatap\tensorindex(\data)=-\E_\tensor\Tdatae\tensorindex(\datatotal)\,.
\end{align*}
This provides the derivative for the second term. Collecting the pieces concludes the proof of the first part.
\paragraph{Part 2} We now compute the second derivative, showing that
\begin{align*}
&\nabla\tensorindex\nabla\tensorindexp \left(\inprodtr{\tensor}{\Tdatae(\datatotala)}+\normtotal(\tensor)\right)\\
&~~~~~~~=\E_\tensor\left[\big(\Tdatae\tensorindex(\datatotal)-\E_\tensor\Tdatae\tensorindex(\datatotal)\big)\big(\Tdatae\tensorindexp(\datatotal)-\E_\tensor\Tdatae\tensorindexp(\datatotal)\big)\right],
\end{align*}
where we again use the shorthand notation $\nabla\tensorindex=\frac{\partial}{\partial \tensor\tensorindex}\,.$

To prove this claim, recall that by Part~1,
\begin{equation*}
  \nabla\tensorindexp\left(\inprodtr{\tensor}{\Tdatae(\datatotala)}+\normtotal(\tensor)\right)=\Tdatae{\tensorindexp}(\datatotala)-\E_\tensor\Tdatae{\tensorindexp}(\datatotal)\,.
\end{equation*}
Since the first term is independent of $\tensor,$ we can focus on the second term. Independence of the observations and the model~\eqref{model} provide
\begin{equation*}
  \E_\tensor\Tdatae\tensorindexp(\datatotal)=  \E_\tensor\Tdatap\tensorindexp(\data)=\int_\Domain \Tdatap\tensorindexp(\datav) e^{-\inprodtr{\tensor}{\Tdatap(\datav)}+\hterm\left(\datav\right)-\normtotal(\tensor)}\,d\measure\,.
\end{equation*}
Taking derivatives, we find similarly as in Part~1
\begin{align*}
  &-\nabla\tensorindex\E_\tensor\Tdatae\tensorindexp(\datatotal)\\
=&-\int_\Domain \Tdatap\tensorindexp(\datav)\nabla\tensorindex e^{-\inprodtr{\tensor}{\Tdatap(\datav)}+\hterm\left(\datav\right)-\normtotal(\tensor)}\,d\measure\\
=&-\int_\Domain \Tdatap\tensorindexp(\datav)\left(-\Tdatap\tensorindex(\datav)-\nabla\tensorindex\normtotal(\tensor) \right) e^{-\inprodtr{\tensor}{\Tdatap(\datav)}+\hterm\left(\datav\right)-\normtotal(\tensor)}\,d\measure\\
=&-\int_\Domain \Tdatap\tensorindexp(\datav)\left(-\Tdatap\tensorindex(\datav)+\E_\tensor\Tdatap\tensorindex(\data) \right) e^{-\inprodtr{\tensor}{\Tdatap(\datav)}+\hterm\left(\datav\right)-\normtotal(\tensor)}\,d\measure\\
=&\int_\Domain\left(\Tdatap\tensorindex(\datav)-\E_\tensor\Tdatap\tensorindex(\data) \right) \left(\Tdatap\tensorindexp(\datav)-\E_\matrix\Tdatap\tensorindexp(\data)\right) e^{-\inprodtr{\tensor}{\Tdatap(\datav)}+\hterm\left(\datav\right)-\normtotal(\tensor)}\,d\measure\\
=&\ \E_\tensor\left[\big(\Tdatap\tensorindex(\data)-\E_\tensor\Tdatap\tensorindex(\data)\big)\big(\Tdatap\tensorindexp(\data)-\E_\tensor\Tdatap\tensorindexp(\data)\big)\right]\\
=&\,n\,\E_\tensor\left[\big(\Tdatae\tensorindex(\datatotal)-\E_\tensor\Tdatae\tensorindex(\datatotal)\big)\big(\Tdatae\tensorindexp(\datatotal)-\E_\tensor\Tdatae\tensorindexp(\datatotal)\big)\right].
\end{align*}
Plugging this in above concludes the proof of Part~2.
\end{proof}

%
%
\pagebreak
\section{Full Algorithm for the maximum likelihood estimator}\label{sec: full algorithm}
In this appendix, we present the full algorithm, which utilizes a backtracking line search to adaptively select step sizes and incorporate the applicable domain constraint. 
By convention, $g(\matrixtoy)$ is infinite for $\matrixtoy \notin \Tensor$. The inequality of backtracking line search implies that $\matrixtoy_{k-1} - \eta\nabla \widetilde{g}(\matrixtoy_{k-1}) \in \Tensor.$ In a practical implementation, we multiple $\eta$ by $\beta$ until $\matrixtoy_{k-1} - \eta\nabla \widetilde{g}(\matrixtoy_{k-1}) \in \Tensor.$

\begin{algorithm}
	\caption{Solving for the maximum likelihood estimator with a backtracking line search}
	\label{full algorithm}
	\SetKwInOut{Input}{Input}
	\SetKwInOut{Output}{Output}
	\tcp{$\eta$ : initial step size}
	\tcp{$\alpha, \beta$: backtracking parameters}
	\Input{$\Tdatae(\datatotal), \eta > 0, \alpha \in (0,0.5), \beta \in (0,1)$}
	\Output{$\tensorest$}
	\tcp{Solve for $\tensor_0$}
	
	$\tensor_0 \leftarrow \mathbf{0}_{p\times p}$\;
	
	\For{$i = 1, \dots, p$}{
		$(\tensor_0)_{ii} \leftarrow \argmin_{m\in\R}\Big\{m\Tdatae_{ii}(\datatotal)+  \log \int \exp \big(-m\Tdatap_{ii}(\datav) + \xi(\dataa_i)\big)d\dataa_i \Big\}$\;
	}
	\tcp{Generate sample set $\mcset$ from $f_{\tensor_0}=\prod_{i=1}^p f_{(\tensor_0)_{ii}}(\dataa_i)$}
	\For{$i = 1, \dots, p$}{
		Generate 10,000 random samples from $f_{(\tensor_0)_{ii}}(\dataa_i)$ for the $i$-th coordinate\;
	}
	\tcp{Apply gradient descent with a backtracking line search to~\eqref{eq: approximate_MLE}}
	$k\leftarrow0$\;
	$\matrixtoy_k\leftarrow\tensor_0$\;
	\Repeat{$\big|\widetilde{g}({\matrixtoy_{k}}) - \widetilde{g}({\matrixtoy_{k-1}})\big| < 10^{-4}$}{
		$k \leftarrow k+1$\;
		$\nabla \widetilde{g}(\matrixtoy_{k-1}) \leftarrow \Tdatae(\datatotal) -  \frac{\sum_{\datamc \in \mcset} \Tdatap(\datamc) e^{-\inprodtr{\matrixtoy_{k-1}-\tensor_0}{\Tdatap(\datamc)}}}{ \sum_{\datamc \in \mcset} e^{-\inprodtr{\matrixtoy_{k-1}-\tensor_0}{\Tdatap(\datamc)}}}$\;
		\tcp{Select the stepsize adaptively using a backtracking line search}
		\Repeat{$ \widetilde{g}(\matrixtoy_{k-1} - \eta \nabla \widetilde{g}(\matrixtoy_{k-1})) \leq  \widetilde{g}(\matrixtoy_{k-1}) - \alpha\eta\,\norm{\nabla \widetilde{g}(\matrixtoy_{k-1})}^2$}{
		$\eta \leftarrow \beta \eta$\;
	}
		$\matrixtoy_{k} \leftarrow \matrixtoy_{k-1} - \eta\nabla \widetilde{g}(\matrixtoy_{k-1})$\;
	}
	$\tensorest \leftarrow \matrixtoy_{k}$\;
\end{algorithm}
\pagebreak
\section{Data Generation}\label{sec:data generation}
Here we describe how to generate samples from an exponential trace model in the form of 
\begin{equation*}
f_\tensor(\datav)=e^{-\inprodtr{\matrix}{\Tdatap(\datav)}+\hterm(\datav)-\normtotal(\matrix)}.
\end{equation*}
Generating data from a multivariate distribution directly is difficult for moderate node numbers. Instead, we consider a Gibbs sampler to sample from the conditional distribution at each iteration. 
Conditioning on all the other variables $\datav_{-j}$, the density of $x_j$ is
\begin{equation*}
f_\matrix(\dataa_{j}\mid \datav_{-j}) \sim e^{-\tensor_{jj}\Tdatap_{jj} - 2\sum_{k\neq j}\tensor_{jk}\Tdatap_{jk}-\hterm(x_j)}.
\end{equation*}
We generate the conditional distribution using a slice sampler~\cite{neal2003slice} with R package \texttt{MfUSampler}~\cite{mahani2014multivariate}. 

Note the slice sampler was designed for continuous variables. For discrete coordinates, we uniformly spread the probability in the spike at an integer $c$ into the interval between $c$ and $c+1$. This defines a continuous density 
\begin{equation*}
	f_\matrix(y_j \mid \datav_{-j})\sim e^{-\tensor_{jj}\lfloor y_j \rfloor - 2\sum_{k\neq j}\tensor_{jk} \Tdatap_{jk}(\lfloor y_j \rfloor,\dataa_{k})-\hterm(\lfloor y_j \rfloor)},
\end{equation*}
where $\lfloor y_j \rfloor$ represents the largest integer less than $y_j$. We sample from the above continuous density and take $\lfloor y_j \rfloor$ as the realization of a discrete $\dataa_{j}$.

\end{document}